\documentclass[12pt]{amsart}
\usepackage{amsfonts,amsmath,amssymb,amscd,amsthm,enumerate}
\usepackage{amsrefs,comment}
\usepackage{geometry}
\usepackage{xcolor}
\usepackage{adjustbox}
\usepackage[all]{xy}
\usepackage{stmaryrd}
\newtheorem{theorem}{Theorem}
\newtheorem{prop}[theorem]{Proposition}
\newtheorem{lem}[theorem]{Lemma}
\newtheorem{cor}[theorem]{Corollary}
\theoremstyle{definition}

\newtheorem{rem}[theorem]{Remark}
\newtheorem{mydef}[theorem]{Definition}
\newtheorem{example}[theorem]{Example}
\AtBeginEnvironment{example}{%
  \pushQED{\qed}%
}\AtEndEnvironment{example}{\popQED}

\renewcommand{\epsilon}{\varepsilon}
\def\lbra{{[}\!{[}}
\def\rbra{{]}\!{]}}

\def\C{\mathbb{C}}

\def\R{\mathbb{R}}

\def\T{\mathbb{T}}

\def\a{\mathbf{a}}

\def\c{\mathbf{c}}

\def\b{\mathbf{b}}
\def\<{\langle}
\def\>{\rangle}

\def\g{\mathfrak{g}}

\def\init{{\rm in}}

\usepackage{multicol} 

\begin{document}
\title[Tensorial Constraints for Commuting Endomorphisms]{Tensorial Constraints for Commuting Endomorphisms of the Generalized Tangent Bundle}
\author{Marco Aldi, Sergio Da Silva, and Daniele Grandini}
\begin{abstract} {In this paper we consider families of mutually commuting endomorphisms of the generalized tangent bundle. We identify natural tensorial constraints extending the notion of a generalized K\"ahler structure to endomorphisms that are not necessarily generalized almost complex structures. These tensors form ideals whose generators we explicitly construct and study using Gr\"obner basis techniques.}
\end{abstract}

\providecommand{\keywords}[1]
{
  \small	
  \textbf{\textit{Keywords---}} #1
}

\keywords{Generalized Geometry, polynomial structures, Gr\"obner bases, Knutson ideals}
\subjclass[2020]{Primary: 53D18; Secondary: 13P10, 13A35, 14P99}
\maketitle

\section{Introduction}

Originally motivated by string theory as a framework to unify $B$-field transformations and diffeomorphisms, the generalized tangent bundle \cite{Hitchin03} has emerged as an indispensable tool in differential geometry. This remarkable geometric object naturally accommodates generalized complex structures \cite{Gualtieri11}, which are integrable skew-symmetric endomorphisms of the generalized tangent bundle that include both complex and symplectic structures as particular cases. Generalized CRF-structures \cite{Vaisman08} provide a broader class of endomorphisms of the generalized tangent bundle and, in addition to generalized complex structures, include odd-dimensional geometries like normal almost contact and cosymplectic structures. Generalized CRF-structures were further extended in \cite{AldiGrandini22} (see also \cites{BlagaNannicini20, EtayoGomezSantamaria22, EtayoGomezSantamaria24}) to generalized polynomial structures: skew-symmetric endomorphisms satisfying a constant-coefficient polynomial equation. A central property of these structures is integrability which specifies their compatibility with the natural Courant-Dorfman bracket. While preliminary notions of integrability for generalized polynomial structures were explored in \cite{AldiGrandini22}, in \cite{AldiDaSilvaGrandini24} we proved that the vanishing of the shifted Courant-Nijenhuis torsion is the strongest tensorial integrability condition that can be imposed uniformly across all generalized polynomial structures, without requiring adaptation to the specific underlying polynomial.

In this paper, we continue the work of \cite{AldiDaSilvaGrandini24} and address the case of several commuting endomorphisms of the generalized tangent bundle. We allow for the possibility that some of these endomorphisms may be skew-symmetric (e.g.\ generalized complex structures) or symmetric (e.g.\ generalized Riemannian metrics) with respect to the tautological inner product on the generalized tangent bundle. The central question that we address is: what possible tensorial conditions can one impose for the compatibility of the commuting endomorphisms with the Courant-Dorfman bracket? Just as in the case of a single endomorphism, we show that these compatibility conditions form an ideal in a suitable polynomial ring. Our main result is the explicit identification of a natural set of generators for this ideal. These generators come in two basic types: cubic and quadratic. The cubic generators can be thought of as direct generalizations of the shifted Courant-Nijenhuis torsion to the case of several commuting endomorphisms. The quadratic generators appear in the ideal only if the commuting family consists of at least two symmetric endomorphisms, and they have no apparent analogue for a single endomorphism. Both types can be constructed out of a natural bivariate version of the Courant-Nijenhuis torsion which we call the {\it semiconcomitant}.  

The proof of our main theorem uses Gr\"obner basis techniques, a crucial tool in commutative algebra used to compute elimination of variables, kernels of ring homomorphisms, ideal membership, projective closures, syzygies, primary decomposition, and normalization, just to name a few \cite{E}. In our case, greatly extending the techniques of \cite{AldiDaSilvaGrandini24}, the use of Gr\"obner bases to compute the intersections of ideals will be especially important. While the Gr\"obner techniques we employ in this manuscript are well-known to algebraists, the direct application of these techniques to the study of the generalized tangent bundle is novel. 

One surprising result is the connection of these compatibility conditions with Frobenius splitting. This positive characteristic property has implications for algebraic varieties defined over a field of characteristic 0, made precise with the notion of a Knutson ideal \cite{Seccia}. In particular, a Frobenius split variety is reduced, has rational singularities, and has certain local cohomology vanishing properties \cite{BK}. Many well-known determinantal varieties, such as Schubert varieties and Kazhdan-Lusztig varieties \cite{BK}, certain $K$-orbit closures \cite{BS}, and some Hessenberg varieties are Frobenius split \cite{CDSHR}.

The paper is organized as follows. In Section \ref{sec:2}, we collect some basic facts about the generalized tangent bundle and Gr\"obner bases to make the paper as self-contained as possible. In Section \ref{sec:3}, we translate the condition of tensoriality for expressions involving the endomorphisms and the Courant-Dorfman bracket into a statement about polynomial membership in a specific ideal. This ideal depends solely on the number of commuting endomorphisms and on their signature, which serves as a bookkeeping device to record which endomorphisms are symmetric and which are skew-symmetric. Section \ref{sec: tensorial} contains our main result, namely the identification of the above mentioned cubic and quadratic generators of this ideal. In Section \ref{sec:5}, we specialize to the case in which all commuting endomorphisms are semisimple. With this assumption we are able to recast the vanishing of the tensorial ideal in terms of involutivity constraints on the common eigenbundles for the commuting endomorphisms. In the Appendix we give an alternate proof, also using Gr\"obner bases, of our main theorem in the special case where all commuting endomorphisms are symmetric.

\section{Preliminaries}\label{sec:2}

\subsection{The Generalized Tangent Bundle}

Throughout the paper, $M$ will be a smooth manifold. We denote by $\Omega_M^k$ the space of smooth $k$-forms on $M$. Recall that the generalized tangent bundle of $M$ is the vector bundle $\T M= TM \oplus T^*M$ i.e.\ the direct sum of the tangent bundle $TM$ and the cotangent bundle $T^*M$ of $M$. The generalized tangent bundle is canonically equipped with the following structures:

\begin{enumerate}[i)]
\item the {\it anchor map}, defined as the projection map $\pi:\T M\rightarrow TM$; 
\item the {\it tautological inner product} $\langle \cdot,\cdot \rangle$, which is the nondegenerate, symmetric $\Omega_M^0$-bilinear map defined by
\begin{equation*}
    \langle X+\alpha, Y+\beta\rangle=\frac{1}{2}\left(\alpha(Y)+\beta(X)\right)
\end{equation*}
for all sections $X,Y\in \Gamma(TM)$ and $\alpha,\beta\in \Omega_M^1$;
\item the {\it Courant-Dorfman bracket}, which is the $\R$-bilinear map
\[
\lbra\cdot ,\cdot \rbra:\Gamma(\T M)\otimes_{\R}\Gamma(\T M)\rightarrow \Gamma(\T M)
\]
defined by requiring
\begin{equation*}
\lbra X+\alpha, Y+\beta\rbra=[X,Y]+{\mathcal L}_X\beta-i_Yd\alpha,
\end{equation*}
for all $X,Y\in \Gamma(TM)$ and $\alpha,\beta\in \Omega_M^1$.
\end{enumerate} 

\noindent Here $\mathcal{L}_X$ is the Lie derivative along the vector field $X$, and $i_Y$ is the interior derivative along $Y$. It is well-known \cite{Gualtieri11} that the data $(\T M,\pi, \langle\cdot ,\cdot \rangle, \lbra\cdot ,\cdot \rbra)$ defines a {\it Courant algebroid}. In particular, the following property is satisfied:
\begin{equation}\label{eq:Courant}\pi(\a)\langle\b,\c\rangle=\langle\lbra\a,\b\rbra,\c\rangle+\langle\lbra\a,\c\rbra,\b\rangle=\langle\lbra\b,\c\rbra,\a\rangle+\langle\lbra\c,\b\rbra,\a\rangle\, ,
\end{equation} 
\noindent for all $\a,\b,\c\in \Gamma(\T M)$, where $\pi(\a)\langle\b,\c\rangle$ denotes the usual action of the vector field $\pi(\a)$ on the smooth function $\langle\b,\c\rangle$. 
Property \eqref{eq:Courant} implies the Leibniz identities
\begin{equation}\label{eq:Leibniz}
\lbra\a,f\b\rbra=f\lbra\a,\b\rbra+\pi(\a)(f)\b, \qquad
\lbra f\a,\b\rbra=f\lbra\a,\b\rbra-\pi(\b)(f)\a+2\langle\a,\b\rangle df
\end{equation}
for all $\a,\b\in \Gamma(\T M)$ and $f\in \Omega_M^0$.
\begin{mydef}  Let $\Theta$ be the $\Omega_M^0$-module of $\R$-trilinear forms
\begin{equation}
\tau:\Gamma(\T M)\otimes_{\mathbb R} \Gamma(\T M)\otimes_{\mathbb R}\Gamma(\T M)\rightarrow \Omega_M^0\,.
\end{equation}
Distinguished among all elements of $\Theta$ is the {\it Courant element}  $\tau_C$ such that \begin{equation}
\tau_C(\a,\b,\c)=\langle\lbra\a,\b\rbra,\c \rangle
\end{equation}
for all $\a,\b,\c\in \Gamma(\T M)$ (see \cite{Courant90}).
\end{mydef}

\subsection{Commuting endomorphisms}
Throughout the paper, $N$ will denote a positive integer.
\begin{mydef}
Let $\epsilon=(\epsilon_1,\dots,\epsilon_N)$ such that $\epsilon_i\in \{\pm 1\}$. A \emph{commuting family of signature} $\epsilon$ is an ordered $N$-tuple $\phi=(\phi_1,\dots,\phi_N)$ such that, for all $i,j\in\{1,\dots, N\}$
\begin{enumerate}[i)]
    \item $\phi_i\in \Gamma({\rm End}(\T M))$;
    \item $\phi_i^*=\epsilon_i\phi_i$;
    \item $\phi_i\phi_j=\phi_j\phi_i$.
\end{enumerate}
Here $\phi^*$ denotes the adjoint with respect to the tautological inner product. Therefore, if $\epsilon_i=1$ (resp.\ $\epsilon_i=-1$) then $\phi_i$ is symmetric (resp.\ skew-symmetric).
\end{mydef}
\begin{example}
Consider the case of two commuting generalized almost complex structures $\phi_1$, $\phi_2$ on $M$. Since $\phi_1$ and $\phi_2$ are both skew-symmetric, $(\phi_1,\phi_2)$ is a commuting family of signature $(-1,-1)$. Equivalently, the same information can be repackaged by defining the symmetric endomorphism $\mathcal G=\phi_1\phi_2$. Then the commuting pair $(\phi_1,\mathcal G)$ has signature $(-1,1)$. 
\end{example}

We will denote the set of all commuting families of signature $\epsilon$ by $\mathcal{F}_{\epsilon}^N$ (or $\mathcal{F}_{\epsilon}$ when $N$ is understood). Note that $\mathcal{F}_{\epsilon}^N$ is an $\Omega_M^0$-module with respect to the usual operations on $N$-tuples.

\subsection{A polynomial action}\label{sec: poly action}
Throughout the paper, for every pair of $N$-tuples $a,b$, we will use the following notation to denote the product:
$$a^b=a_1^{b_1}a_2^{b_2}\dots a_N^{b_N}.$$
Consider the ring $$\mathcal{R}=\mathcal{R}_N(x,y,z):=\R[x_{1},\ldots, x_N, y_{1},\ldots, y_N, z_{1},\ldots, z_N].$$
Using the above convention, elements of $\mathcal R$ will be  expressed in the form
\begin{equation}\label{eq:5bis}
P(x,y,z)=\sum_{I,J,K}a_{I,J,K}x^Iy^Jz^K\,,
\end{equation} 
where
\begin{enumerate}[i)]
    \item $I,J,K$ are ordered $N$-tuples of non negative integers:
    $$I=(i_1,\dots,i_N), \qquad J=(j_1,\dots,j_N),\qquad K=(k_1,\dots,k_N);$$
    \item $x, y, z$ are ordered $N$-tuples of variables:
    $$x=(x_1,\dots, x_N), \qquad y=(y_1,\dots, y_N), \qquad z=(z_1,\dots, z_N);$$
    \item $a_{IJK}\in \R$ for all $I,J,K$.
\end{enumerate}
 Every commuting family $\phi\in\mathcal{F}_{\epsilon}$ induces an $\R$-linear action on the module $\Theta$:
\begin{equation}\label{eq:6}
\bullet_{\phi}:\mathcal{R}\times \Theta\longrightarrow \Theta,\qquad\qquad (P,\tau)\longmapsto P\bullet_{\phi}\tau\,,
\end{equation}
uniquely determined by requiring
\begin{align*}
(x_i\bullet_{\phi} \tau)(\a,\b,\c)&=\tau(\phi_i\a,\b,\c)\,;\\(y_i\bullet_{\phi} \tau)(\a,\b,\c)&=\tau(\a,\phi_i\b,\c)\,;\\(z_i\bullet_{\phi} \tau)(\a,\b,\c)&=\tau(\a,\b,\phi_i\c)
\end{align*}
for all $\tau\in\Theta$. The following propositions summarize how the polynomial action \eqref{eq:6} interacts with the action of the symmetric group $\mathcal{R}$ and with linear combinations.
\begin{prop}\label{prop: poly properties} Fix the usual left action of the symmetric groups $S_n$ on ordered $n$-tuples, and define the following actions:
\begin{enumerate}[a)]
\item for all $P\in \mathcal{R}$ and $\sigma\in S_3$, let $\sigma P$ be the polynomial obtained by applying $\sigma$ to the triple of variables $(x_i,y_i,z_i)$ for all $i\in\{1,\dots,N\}$;
\item for all $P\in \mathcal{R}$ and $\rho\in S_N$, let $P_{\rho}$ be the polynomial
$$P_{\rho}(x,y,z)=P(\rho x, \rho y, \rho z);$$

\item for all $\tau\in \Theta$, let $(\sigma \tau)(\a,\b,\c)=\tau\left(\sigma^{-1}(\a,\b,\c)\right)$ for all $\a,\b,\c\in \Gamma(\T M)$.
\end{enumerate}
Then, we have
\begin{enumerate}[i)]
\item $\displaystyle\sigma (P\bullet_{\phi} \tau)=(\sigma P)\bullet_{\phi} (\sigma \tau)$;
\item $\displaystyle P\bullet_{\rho\phi} \tau=P_{\rho}\bullet_{\phi} \tau,$
\end{enumerate}
for all $\phi\in\mathcal{F}_{\epsilon}$, $P\in \mathcal{R}$,  and $\tau\in \Theta$.
\end{prop}
\begin{proof} For the proof of (i), it is sufficient to set $P=x_i$ and check the statement on generators $\sigma=(12)$ or $\sigma=(123)$ of $S_3$. If $\sigma=(12)$, we have
$$\left(\sigma (x_i\bullet_{\phi} \tau)\right)(\a,\b,\c)=\left(x_i\bullet_{\phi} \tau\right)(\b,\a,\c)=\tau(\phi_i\b,\a,\c) $$
while
$$\left((\sigma x_i)\bullet_{\phi} (\sigma \tau)\right)(\a,\b,\c)=\left(y_i\bullet_{\phi} (\sigma \tau)\right)(\a,\b,\c)=(\sigma \tau)(\a,\phi_i\b,\c)=\tau (\phi_i\b,\a,\c).$$
If  $\sigma=(123)$, we have instead
$$\left(\sigma (x_i\bullet_{\phi} \tau)\right)(\a,\b,\c)=\left(x_i\bullet_{\phi} \tau\right)(\b,\c,\a)=\tau(\phi_i\b,\c,\a) $$
while
$$\left((\sigma x_i)\bullet_{\phi} (\sigma \tau)\right)(\a,\b,\c)=\left(y_i\bullet_{\phi} (\sigma \tau)\right)(\a,\b,\c)=(\sigma \tau)(\a,\phi_i\b,\c)=\tau (\phi_i\b,\c,\a)\,.$$
Let $P$ be as in \eqref{eq:5bis}. Since
\begin{align*}
\left(P\bullet_{\rho\phi} \tau\right)(\a,\b,\c)&=\sum_{I,J,K}a_{I,J,K}\tau((\rho\phi)^I\a,(\rho\phi)^J\b,(\rho\phi)^K\c)\\
&=\sum_{I,J,K}a_{I,J,K}\tau(\phi^{\rho^{-1}I}\a,\phi^{\rho^{-1}J}\b,\phi^{\rho^{-1}K}\c)\\
&=\sum_{I,J,K}a_{\rho I,\rho J,\rho K}\tau(\phi^{I}\a,\phi^{J}\b,\phi^{K}\c)\\
&=\left(P_{\rho}\bullet_{\phi} \tau\right)(\a,\b,\c)
\end{align*}
we obtain ii).
\end{proof}
In what follows, for any $N$-tuple $a$ and any $N'$-tuple $a'$ we will denote by $aa'$ the $N+N'$-tuple given by the juxtaposition of $a$ and $a'$. Moreover, if $N=N'$ we will denote by $a+a'$ the sum of the $N$-tuples.
Given two families $\phi\in \mathcal{F}_{\varepsilon}^N$ and $\phi'\in \mathcal{F}_{\varepsilon'}^{N'}$, consider their actions
$$\bullet_{\phi}:{\mathcal{R}_N(t,u,v)\times\Theta\rightarrow\Theta}, \qquad \bullet_{\phi'}:{\mathcal{R}_{N'}(t',u',v')\times\Theta\rightarrow\Theta}.$$
If $\phi\phi'\in \mathcal{F}_{\varepsilon\varepsilon'}^{N+N'}$ (that is, if $\phi_i\phi_j'=\phi_j'\phi_i$ for all indices $i,j$), these two actions can be combined into their \emph{composite}, that is the action 
$$\bullet_{\phi,\phi'}:\mathcal{R}_{N+N'}(t'',u'',v'')\times\Theta\rightarrow\Theta$$
where $t''=tt'$, $u''=uu'$, and $v''=vv'$, uniquely defined as follows: for all $P\in \mathcal{R}_N(t,u,v)$ and all $P'\in \mathcal{R}_{N'}(t',u',v')$, and $\tau\in \Theta$, we have
$$PP'\bullet_{\phi,\phi'}\tau=P\bullet_{\phi}\left(P'\bullet_{\phi'}\tau\right).$$
It is straightforward to check, on the monomials in the variables $t'',u'',v''$, that the action coincides with $\bullet_{\phi\phi'}$; that is, the action of $\mathcal{R}_{N+N'}(t'',u'',v'')$ on $\Theta$ associated to the juxtaposition of $\phi$ and $\phi'$.

\begin{prop}\label{prop: properties} Let $\phi,\phi'\in\mathcal{F}_{\epsilon}^N$ such that $\phi\phi'\in \mathcal{F}_{\epsilon\epsilon}^{2N} $, and let $c\in\R^N$. Then, for all $P\in{\mathcal R}_N$ and $\tau\in \Theta$,
\begin{enumerate}[i)]
\item $P\bullet_{{\rm diag}(c)\phi}\tau = Q\bullet_\phi \tau$ where
$${\rm diag}(c)\phi=(c_1\phi_1, c_2\phi_2, \dots ,c_N\phi_N)\,;$$
and
\[
Q(x,y,z)=P({\rm diag}(c)x,{\rm diag}(c)y,{\rm diag}(c)z)\,.
\]
\item 

$\displaystyle P\bullet_{\phi+\phi'}\tau=P(t+t',u+u',v+v')\bullet_{\phi,\phi'}\tau$,
where the latter action is the composite of the actions
$$\bullet_{\phi}:{\mathcal{R}_N(t,u,v)\times\Theta\rightarrow\Theta}, \qquad \mbox{ and }\qquad\bullet_{\phi'}:{\mathcal{R}_{N'}(t',u',v')\times\Theta\rightarrow\Theta}.$$

\end{enumerate}

\end{prop}

\begin{proof} Let $P$ be as in \eqref{eq:5bis}. Since
\begin{align*}
\left(P\bullet_{{\rm diag}(c)\phi}\tau\right)(\a,\b,\c) &=\sum_{I,J,K} a_{I,J,K}\cdot\tau(({\rm diag}(c)\phi)^I\a, ({\rm diag}(c)\phi)^J\b, ({\rm diag}(c)\phi)^K\c )\\
&=\sum_{I,J,K} a_{I,J,K}\cdot\tau(c^I\phi^I\a, c^J\phi^J\b, c^K\phi^K\c )\\
&=\sum_{I,J,K} a_{I,J,K}c^{I+J+K}\cdot\tau(\phi^I\a, \phi^J\b, \phi^K\c )\\
&=\left(Q\bullet_{\phi}\tau\right)(\a,\b,\c)
\end{align*}
statement i) holds. 
Also,
$$\left(P\bullet_{\phi+\phi'}\tau\right)(\a,\b,\c)=\sum_{I,J,K} a_{I,J,K}\cdot\tau((\phi+\phi')^I\a, (\phi+\phi')^J\b, (\phi+\phi')^K\c )=$$
$$=\sum_{I,J,K,I_1,J_1, K_1}{I\choose I_1}{J\choose J_1}{K\choose K_1}   a_{I,J,K}\cdot\tau(\phi^{I_1}\phi'^{I-I_1}\a, \phi^{J_1}\phi'^{J-J_1}\b, \phi^{K_1}\phi'^{K-K_1}\c )=$$
$$=\left(P(t+t',u+u',v+v')\bullet_{\phi,\phi'}\tau\right)(\a,\b,\c).$$
\end{proof}

\subsection{Multi-homogeneous polynomials}
For a monomial $x^Iy^Jz^K\in \mathcal{R}_N(x,y,z)$, the $N$-tuple $I+J+K$ will be called its $N$-degree (or multidegree). 
For an $N$-tuple $D$, let $\mathcal{R}^D$ denote the subspace generated by the monomials of fixed multidegree $D$. (Elements of $\mathcal{R}^D$ will be called multi-homogeneous). Then, the ring $\mathcal{R}$ is equipped  with a multigrading
$$\mathcal{R}=\bigoplus_D \mathcal{R}^D.$$
Note that each subspace $\mathcal{R}^D$ is $S_3$-invariant, in the sense of Proposition \ref{prop: poly properties}. As we will see, the action of $\mathcal{R}$ from Section \ref{sec: poly action} will allow us to define $S_3$-invariant ideals $\mathcal{I}\subseteq \mathcal{R}$ that are preserved by the multigrading, i.e. such that
$$\mathcal{I}=\bigoplus_D (\mathcal{I}\cap \mathcal{R}^D).$$

\subsection{Gr\"obner Bases}\label{subsec: Grobner}
Many of the arguments in Section \ref{sec: tensorial} and Appendix \ref{app: alternate} will involve techniques from commutative algebra regarding \emph{Gr\"obner bases}. In this section, we will briefly recall some of the fundamentals regarding Gr\"obner bases and refer the reader to \cite{CLO,E,EH} for further details and background. 

Let $R=\mathbb{K}[u_1,\ldots,u_r]$ be a polynomial ring with $\mathbb{K}$ any field, and let $<$ be a \emph{monomial order} on $R$. A monomial order, which is a total order on the monic monomials of $R$ respecting monomial multiplication, allows us to distinguish a leading monomial for any polynomial $f\in R$. More specifically, given a polynomial $f$, we will denote by $\init_<(f)$ the term of $f$ which is greater than all of the other terms of $f$ with respect to $<$. As we do this for every polynomial in an ideal $I$ of $R$, we obtain the \textit{initial ideal} of $I$ defined as the monomial ideal \[\init_<(I) = \langle \init_<(f)|f\in I\rangle\,.\] 

Many important algebro-geometric properties may be studied using $\init_<(I)$ instead of $I$ itself. For example, if $\init_<(I)$ is a squarefree monomial ideal, then $I$ is a radical ideal. In fact, many properties that are true for $\init_<(I)$ (like being \emph{Gorenstein, Cohen-Macaulay, smooth, factorial}, etc.) also hold for $I$. Since proving results for monomial ideals is often far easier than for general ideals, a common approach is to consider a problem from this perspective (often called a \emph{Gr\"obner degeneration}).   Furthermore, combinatorial representations of monomial ideals, like using simplicial complexes arising from the \emph{Stanley-Reisner correspondence}, allow for combinatorial descriptions of ideal properties.

By Hilbert's basis theorem $\init_<(I)$ is finitely generated, leaving the question of how to describe a finite generating set for the ideal. This motivates the following definition.

\begin{mydef} Let $I\subset R$ be an ideal and $<$ a monomial order on $R$. Suppose that $\mathcal{G} =\{ f_1,\ldots, f_k \}$ is a generating set of $I$. We say that $\mathcal{G}$ is a Gr\"obner basis of $I$ with respect to $<$ if \[\init_<(I) = \langle \init_<(f_1), \ldots, \init_<(f_k)\rangle.\]\end{mydef}

\begin{example}
Let $I=\langle x-y,x-z\rangle \subset \mathbb{K}[x,y,z]$, and let $<$ be the lexicographic order $x>y>z$. Then $\init_<(x-y)=\init_<(x-z)=x$, but it is not true that $\init_<(I)= \langle x\rangle$. Indeed, $y-z = (x-z)-(x-y)\in I$, but $\init_<(y-z)=y$.

This means that $\{x-y,x-z\}$ is not a Gr\"obner basis of $I$ with respect to $<$. If we however extend this list by including $y-z$, then we do get a Gr\"obner basis. 
\end{example}

Given a generating set of $I$ and a monomial order $<$, we can always compute a Gr\"obner basis for $I$, using \emph{Buchberger's algorithm} for instance (computationally faster techniques exist \cite{BFS}, but Buchberger's will be the most convenient for proving results). In the previous example, an issue arose where the cancellation of initial terms of $x$ in $x-y$ and $x-z$ resulted in a polynomial $y-z$ with a smaller initial term with respect to $<$. To systematically ensure that all such cancellation is considered, one arrives at the notion of an $S$-polynomial (also referred to as an $S$-pair). Given two polynomials $f$ and $g$ of $R$, we define  \[S(f,g) := \frac{\init_<(g)}{\mathrm{GCD}(\init_<(f),\init_<(g))}f - \frac{\init_<(f)}{\mathrm{GCD}(\init_<(f),\init_<(g))}g.\]

\noindent Buchberger's algorithm can be summarized as follows:

\begin{enumerate}[i)]
\item Given $I=\langle f_1,\ldots, f_k\rangle$, compute $S(f_i,f_j)$ for all pairs of generators with $i<j$. 
\item Using multivariate polynomial division with respect to $<$ and a fixed ordering of the list of generators $\{f_1,\ldots, f_k\}$, write $S(f_i,f_j) = \sum a_tf_t +r_{i,j}$ where no $\init_<(f_i)$ divides any term of the remainder $r_{i,j}$. This is called a \emph{standard form} for $S(f_i,f_j)$.
\item The generating set $\{f_1,\ldots, f_k\}$ is a Gr\"obner basis with respect to $<$ if and only if all of the $r_{i,j}=0$. 
\item For all $i,j$ for which $r_{i,j}\neq 0$, append $r_{i,j}$ to the list of generators $\{f_1,\ldots, f_k\}$, and repeat the process with the expanded list.
\end{enumerate}

\noindent The algorithm always terminates after finitely many steps since the terms of $r_{i,j}$ are strictly smaller than the initial terms of the current generating set.

\section{Tensoriality}\label{sec:3}
Given a polynomial $P\in \mathcal R$, a commuting family $\phi$ and an $\R$-trilinear form $\tau$ on $\T M$, the polynomial action  in Section \ref{sec: poly action} allows us to define a new $\R$-trilinear form $P\bullet_{\phi}\tau$. Note that, if $\tau$ is $\Omega_M^0$-linear (or equivalently, a section of the vector bundle $(\T M\otimes \T M \otimes \T M)^*$), then so is $P\bullet_{\phi}\tau$. In this section we focus on the case when $\tau$ is not $\Omega_M^0$-linear (specifically for $\tau=\tau_C$, the Courant element) and look for conditions on the pair $(P,\phi)$ ensuring that $P\bullet_{\phi}\tau$ is $\Omega_M^0$-linear.

\begin{mydef}\label{def:7}
Let $\phi\in\mathcal{F}_{\epsilon}$, and $P\in \mathcal{R}$. The pair $(P,\phi)$ is called \textit{tensorial} if
$$P\bullet_{\phi} \tau_C\in \Gamma((\T M\otimes \T M\otimes \T M)^*).$$
Moreover, a polynomial $P\in \mathcal{R}$ is said to be \textit{universally  tensorial of signature $\epsilon$} if $(P,\phi)$ is a tensorial pair for all $\phi\in\mathcal{F}_{\epsilon}$ and all smooth manifolds $M$.
\end{mydef}

\begin{example}
Definition \ref{def:7} generalizes the notion of a tensorial pair for a single endomorphism $\phi=(\phi_1)$ introduced in \cite{AldiGrandini22}. In that case the Courant-Nijenhuis torsion $\mathcal T_{\phi_1}$ satisfies 
\[
\langle \mathcal T_{\phi_1}(\a,\b),\c\rangle=(P\bullet_\phi \tau_C)(\a,\b,\c)
\]
for all $\a,\b,\c\in \Gamma(\mathbb TM)$
with $P=(x+z)(y+z)$. If $\phi_1$ is a generalized almost complex structure, then $(P,\phi)$ is tensorial even though $P$ is not universally tensorial of signature $-1$. On the other hand, consider the {\it shifted Courant-Nijenhuis torsion} $\mathcal S_{\phi_1}$ defined by
\[
\langle \mathcal S_{\phi_1}(\a,\b),\c\rangle = \langle \mathcal T_{\phi_1}(\phi_1\a,\b)+\mathcal T_{\phi_1}(\a,\phi_1\b),\c \rangle=  (S\bullet_\phi \tau_C)(\a,\b,\c)
\]
for all $\a,\b,\c\in \Gamma(\mathbb TM)$ with $S=(x+z)(y+z)(x+y)$. Then $S$ is universally tensorial of signature $-1$. 
\end{example}

\begin{example}\label{ex:9}
Let $\phi=(\phi_1,\phi_2)$ be a commuting pair of endomorphisms of the generalized tangent bundle. Then
\begin{equation}\label{eq:5}
\mathcal K_\phi(\a,\b)=\lbra \phi_1\a, \phi_2\b\rbra-\phi_1\lbra \a,\phi_2\b\rbra - \phi_2\lbra \phi_1\a,\b\rbra +\phi_1\phi_2\lbra \a,\b\rbra
\end{equation}
 satisfies
\[
\langle \mathcal K_\phi(\a,\b),\c\rangle = (P\bullet_\phi \tau_C)(\a,\b,\c)
\]
with $P=(x_1-\varepsilon_1 z_1)(y_2-\varepsilon_2 z_2)$. While $P$ is not universally tensorial of signature $\varepsilon$, $(P,\phi)$ is a tensorial pair under the additional assumption that $\phi_1=\phi_2$ is a generalized almost complex structure, in which case $\mathcal K_\phi$ reduces to the Courant-Nijenhuis torsion of $\phi_1$. The symmetrization 
\[
\mathcal K_{(\phi_1,\phi_2)}+\mathcal K_{(\phi_2,\phi_1)}=\mathcal T_{\phi_1+\phi_2}-\mathcal T_{\phi_1}-\mathcal T_{\phi_2}
\]
coincides with the {\it Nijenhuis concomitant} studied in \cite{AntunesNunes13} and with the bivariate Courant-Nijenhuis torsion discussed in \cite{AldiGrandini22}. For this reason we propose the terminology \textit{semiconcomitant} for $\mathcal K_\phi$. 
\end{example}

\noindent In light of Definition \ref{def:7}, we can consider two natural ideals of $\mathcal{R}$:

\begin{enumerate}[i)]

\item ${\mathcal I}(\phi)$, consisting of all $P\in \mathcal{R}$ such that $(P,\phi)$ is a tensorial pair;
\item ${\mathcal I}_{\epsilon}$, consisting of the
universally tensorial polynomials of signature $\epsilon$.
\end{enumerate}
In the $N=1$ case, the ideal $\mathcal I(\varphi)$ was introduced in \cite{AldiGrandini22}. We will now consider the general case.

\begin{prop}\label{prop:tens} Let $P$ be as in \eqref{eq:5bis} and let  $\phi\in\mathcal{F}_{\epsilon}$. The pair $(P,\phi)$ is tensorial 
if and only if the following conditions are satisfied:
\begin{eqnarray}\label{eq:tens}&&\sum_{I,J,T}\epsilon^Ja_{I,J,T-J}\pi\left(\phi^I(\a)\right)(f)\phi^{T}=0\nonumber\\ \\
&&\sum_{I,J,T}\left(\epsilon^{T-J}a_{T-J,I,J}\pi\left(\phi^I(\b)\right)(f)\phi^{T}(\c)\right)=  \sum_{I,J,T}\left(\epsilon^J    a_{J,T-J,I}\pi\left(\phi^I(\c)\right)(f)\phi^{T}(\b)\right)\,.\nonumber\end{eqnarray}
\end{prop}
\begin{proof}
We adapt the proof of Proposition 8 in \cite{AldiDaSilvaGrandini24}. Given $f\in\Omega_M^0$,
$$(P\bullet_{\phi}\tau_C)(\a,\b,f\c)-f(P\bullet_{\phi}\tau_C)(\a,\b,\c)$$
is equal to
\[
\sum_{I,J,K}a_{I,J,K}\langle\lbra\phi^I\a,\phi^J\b\rbra,\phi^K(f\c)\rangle-f\sum_{I,J,K}a_{I,J,K}\langle\lbra\phi^I\a,\phi^J\b\rbra,\phi^K\c\rangle
\]
which vanishes by the $\Omega_M^0$-linearity of $\phi^K$ and of the tautological inner product. Hence, $P\bullet_{\phi}\tau_C$ is always $\Omega_M^0$-linear in the third variable. Moreover,
\begin{align*}
(P\bullet_{\phi}\tau_C)(\a,f\b,\c)-f(P\bullet_{\phi}\tau_C)(\a,\b,\c)&=\sum_{I,J,K}a_{I,J,K}\pi(\phi^I\a)(f)\langle\phi^J\b,\phi^K\c\rangle\\
&=\sum_{I,J,K}\epsilon^Ja_{I,J,K}\pi(\phi^I\a)(f)\langle\b,\phi^{J+K}\c\rangle\,.
\end{align*}
Using the substitution $T=J+K$, the above sum becomes
$$=\sum_{I,J,T}\epsilon^{J}a_{I,J,T-J}\pi(\phi^I\a)(f)\langle\b,\phi^{T}\c\rangle$$
which is equal to $0$ for all $\a,\b,\c\in \Gamma(\T M)$ if and only if
$$\sum_{I,J,T}\epsilon^{J}a_{I,J,T-J}\pi(\phi^I\a)(f)\phi^{T}=0$$
for all $\a\in \Gamma(\T M)$. The calculation for the $\Omega_M^0$-linearity in the first variable is similar.

\end{proof}

\begin{rem}
The above construction can be extended to polynomials with nonconstant coefficients; that is, coefficients which are elements of the $\Omega_M^0$-module
$$\mathcal{R}(M):=\Omega_M^0\otimes_{\R}\mathcal{R}.$$
In fact, the $\R$-linear action of $\mathcal{R}$ on $\Theta$ extends uniquely to an action of the module $\mathcal{R}(M)$ on $\Theta$ such that, for all $f\in \Omega_M^0$, $P\in \mathcal{R}$ and $\tau\in\Theta$, 
$$(f\otimes P)\bullet\tau=f\cdot(P\bullet\tau).$$
The notion of tensorial couple can be extended verbatim to polynomials in $\mathcal{R}(M)$. However, the notion of \lq\lq universal tensoriality\rq\rq\ cannot be transported to this extended context, but can be replaced by the following condition of \textit{pointwise universal tensoriality}:
$$P_p\in \mathcal{I}_{\epsilon}\mbox{ for all }p\in M\,.$$
It is straightforward to see that pointwise universally tensorial polynomials are simply the elements of 
$$\Omega_M^0\otimes{\mathcal{I}_{\epsilon}}.$$
\end{rem}

\begin{theorem}\label{theo:unitens} Let $P=\sum_{I,J,K}a_{IJK}x^Iy^Jz^K \in \mathcal{R}$. The following conditions are equivalent:
\begin{enumerate}[i)]
    \item  $P$ is universally tensorial;
    \item  the following system of equations is satisfied:
\begin{equation}\label{eq:unitens}\sum_{J}\epsilon^Ja_{I,J,T-J}=0,\quad \sum_{J}\epsilon^Ja_{J,T-J,I}=0, \quad \sum_{J}\epsilon^Ja_{T-J,I,J}=0\end{equation}
for all $I,T$;
\item the polynomial $P$ vanishes on the following subvariety of $\C^{3N}$:
$$V_{\epsilon}:=\left\{(x,y,z):y_i=\epsilon_iz_i\ \forall i\right\}\cup \left\{(x,y,z):z_i=\epsilon_ix_i\ \forall i\right\}\cup\left\{(x,y,z):x_i=\epsilon_iy_i\ \forall i\right\}\,.$$
\end{enumerate}
\end{theorem}
\begin{proof} We adapt the proof of Theorem 9 in \cite{AldiDaSilvaGrandini24}. By substituting the condition of iii) into the expression for $P$ we recover ii). Moreover, ii) coupled with Proposition \ref{prop:tens} implies i). We are left to show that if $P$ is universally tensorial, then it vanishes when restricted to $V_\varepsilon$. It suffices to look at a particular manifold $M$ endowed with a commuting family $\phi$ of endomorphisms of the generalized tangent bundle of signature $\varepsilon$ such that each $\phi_i$ maps vectors to vectors and with linearly independent vector fields $X,Y$ such that $\phi_i(X)=\lambda_i X$ and $\phi_i(Y)=\mu_i Y$ for all $i\in \{1,\ldots,N\}$. Then, by Proposition \ref{prop:tens}, $P(\lambda,{\rm diag}(\varepsilon) \mu, \mu)X(f)Y=0$ for all $f$ and thus $P(\lambda,{\rm diag}(\varepsilon) \mu, \mu)=0$. The remaining cases are proved in a similar way.

\end{proof}

It follows from equations \eqref{eq:unitens} that the ideal ${\mathcal I}_{\epsilon}$ is $S_3$-invariant. Also, it is preserved by the multi-grading, i.e.,\
$${\mathcal I}_{\epsilon}=\bigoplus_D {\mathcal I}^D_{\epsilon}$$
where ${\mathcal I}^D_{\epsilon}={\mathcal I}_{\epsilon}\cap \mathcal{R}^D$. It is also clear that the subspaces ${\mathcal I}^D_{\epsilon}$ are $S_3$-invariant as well.

\begin{rem}\label{rem: alternating} From Theorem \ref{theo:unitens} follows that if  $P\in\mathcal{I}_{\varepsilon}$  and  additionally $P$ is invariant under the permutations of the letters $x,y,z$ (that is, $\sigma P=P$ for all $\sigma\in S_3$), then all tensors resulting from it are necessarily alternating. In fact, for all $\a,\b,\c\in \Gamma(\T M)$, define the $\R$-trilinear function $\theta(\a,\b,\c)=\pi(\a)\langle\b,\c\rangle$.
If $P$ is as in \eqref{eq:5bis} and $\phi\in\mathcal{F}_{\varepsilon}$, then we have
\begin{equation}
\left(P\bullet_\phi\theta\right)(\a,\b,\c)=\sum_{I,J,T}\epsilon^Ja_{I,J,T-J}\pi\left(\phi^I(\a)\right)\langle\b, \phi^T(\c)\rangle\,.
\end{equation}
On the other hand, since $P$ is universally tensorial we can invoke Theorem \ref{theo:unitens} and obtain
\begin{equation}\left(P\bullet_\phi\theta\right)(\a,\b,\c)=\sum_{I,T}\left(\sum_J\epsilon^Ja_{I,J,T-J}\right)\pi\left(\phi^I(\a)\right)\langle\b, \phi^T(\c)\rangle=0.\nonumber\end{equation}
Now, recall that
\[
\left(P\bullet_\phi\tau_C\right) (\a,\b,\c) = \sum_{I,J,K}a_{I,J,K}\langle\lbra\phi^I\a, \phi^J\b\rbra,\phi^K\c\rangle\,,
\]
so that from the identities (\ref{eq:Courant}) and the invariance of $P$ follows that the quantity
\[
\left(P\bullet_\phi\tau_C\right) (\a,\b,\c)+\left(P\bullet_\phi\tau_C\right)(\a,\c,\b)
\]
is equal to
\[ \sum_{I,J,K}a_{I,J,K}\left(\langle\lbra\phi^I\a, \phi^J\b\rbra,\phi^K\c\rangle+\langle\lbra\phi^I\a, \phi^K\c\rbra,\phi^J\b\rangle\right)=\left(P\bullet_\phi \theta\right) (\a,\b,\c)\,.
\]
Similarly, 
\[
\left(P\bullet_\phi\tau_C \right)(\a,\b,\c)+\left(P\bullet_\phi\tau_C\right) (\b,\a,\c)=\left(P\bullet_\phi \theta\right) (\c,\a,\b)\,,
\]
which is sufficient to conclude that the tensor $P\bullet_{\phi}\tau_C$ is alternating. \hfill \qed
\end{rem}

\begin{example}\label{ex:16}
Given a commuting family $\phi\in {\mathcal F}_{\epsilon}$ and $i,j,k\in \{1,\ldots,N\}$, consider the polynomial 
\begin{equation}\label{eq:11}
T^{ijk}_\varepsilon=(x_i-\varepsilon_iy_i)(y_j-\varepsilon_jz_j)(z_k-\varepsilon_kx_k)\,.
\end{equation}
By Theorem \ref{theo:unitens}, $T^{ijk}_\varepsilon$ is in the universally tensorial ideal $\mathcal I_\varepsilon$. Moreover,
\[
\widetilde{{\mathcal T}}^{ijk}_{\phi}(\a,\b,\c):=(T^{ijk}_{\epsilon}\bullet_{\phi}\tau_C)(\a,\b,\c)=\langle {\mathcal T}^{ijk}_{\phi}(\a,\b),\c\rangle 
\]
for all sections $\a,\b,\c$ of the generalized tangent bundle, where
\begin{align*}
{\mathcal T}^{ijk}_{\phi}(\a,\b)&=\varepsilon_k\phi_k\lbra\phi_i\a,\phi_j\b\rbra-\varepsilon_k\epsilon_i\phi_k\lbra\a,\phi_i\phi_j\b\rbra+\phi_j\phi_k\lbra\phi_i\a,\b\rbra+\varepsilon_k\epsilon_i\phi_j\phi_k\lbra\a,\phi_i\b\rbra+\\
& -\varepsilon_k\lbra\phi_i\phi_k\a,\phi_j\b\rbra+\varepsilon_k\epsilon_i\lbra\phi_k\a,\phi_i\phi_j\b\rbra+\varepsilon_k\phi_j\lbra\phi_i\phi_k\a,\b\rbra-\varepsilon_k\epsilon_i\phi_j\lbra\phi_k\a,\phi_i\b\rbra\,.
\end{align*}
Using the semiconcomitant introduced in Example \ref{ex:9}, we can equivalently write
\begin{equation}\label{eq:10}
{\mathcal T}^{ijk}_{\phi}(\a,\b)=\varepsilon_i\varepsilon_k\mathcal K_{(\phi_k,\phi_j)}(\a,\phi_i\b)-\varepsilon_k \mathcal K_{(\phi_k,\phi_j)}(\phi_i\a,\b)\,.
\end{equation}
If $N=1$, we obtain only one possible expression, namely, 
\[
\mathcal T^{111}_\phi(\a,\b)=\mathcal T_{\phi_1}(\a,\phi_1\b)-\varepsilon_1\mathcal T_{\phi_1}(\phi_1\a,\b)
\]
for all $\a,\b\in \Gamma(\mathbb TM)$. If $\varepsilon_1=-1$, we recover the shifted Courant-Nijenhuis torsion $\mathcal S_{\phi_1}$.

If $N=2$, then we have eight possible ``torsions'' $\mathcal T_\phi^{lll}$, $\mathcal T_\phi^{mll}$, $\mathcal T_\phi^{lml}$, $\mathcal T_\phi^{llm}$, for all $l,m\in \{1,2\}$, $l\neq m$. Moreover,  $\mathcal T_\phi^{lll}=-\mathcal S_{\phi_l}$ and   
\begin{equation}
\mathcal T_\phi^{mll}(\a,\b)=\varepsilon_l \varepsilon_m \mathcal T_{\phi_l}(\a,\phi_m\b) - \varepsilon_l \mathcal T_{\phi_l}(\phi_m\a,\b)
\end{equation}
for all $\a,\b\in \Gamma(\T M)$. In the particular case in which $\phi_1$ and $\phi_2$ are both generalized almost complex structures, 
\[
\mathcal T_\phi^{lml}(\a,\b)=\mathcal T_{\phi_l}(\a,\phi_m\b)-\phi_m\mathcal T_{\phi_l}(\a,\b)=\mathcal T_\phi^{llm}(\a,\b)\,,
\]
from which we conclude that $\phi_1$ and $\phi_2$ are both integrable (i.e.\ generalized complex structures) if and only if all the eight torsions $\mathcal T_\phi^{ijk}$ vanish identically.

\end{example}
\begin{example}\label{ex:18}
Now, let $i,j\in \{1,\ldots,N\}$ such that $\varepsilon_i=\varepsilon_j=1$ and consider the polynomial 
\begin{equation}\label{eq:14}
P^{ij}:=(z_i-x_i)(x_j-y_j)-(z_j-x_j)(x_i-y_i).
\end{equation}
By Theorem \ref{theo:unitens}, this polynomial is in the universally tensorial ideal $\mathcal{I}_{\varepsilon}$.
Moreover,
$$\widetilde{{\mathcal P}}^{ij}_{\phi}(\a,\b,\c):=\left(P^{ij}\bullet_{\phi}\tau_C\right)(\a,\b,\c)=\langle\mathcal{P}_{\phi}^{ij}(\a,\b),\c\rangle,$$
where
\begin{equation}
\mathcal{P}_{\phi}^{ij}(\a,\b)=\lbra\phi_i\a,\phi_j\b\rbra-\lbra\phi_j\a,\phi_i\b\rbra-\phi_j\lbra\phi_i\a,\b\rbra+\phi_i\lbra\phi_j\a,\b\rbra+\phi_j\lbra\a,\phi_i\b\rbra-\phi_i\lbra\a,\phi_j\b\rbra\,.
\end{equation}
Equivalently, $\mathcal P_\phi^{ij}=\mathcal K_{(\phi_i,\phi_j)}-\mathcal K_{(\phi_j,\phi_i)}$. In particular, it is clear that $\mathcal P_\phi^{ii}=0$ for every $i\in \{1,\ldots,N\}$.

\end{example}
In the next section we will prove that the polynomials $T^{ijk}_{\varepsilon}$ and $P^{ij}$ considered in Examples \ref{ex:16} and \ref{ex:18} indeed generate the ideal $\mathcal{I}_{\varepsilon}$. This means that every trilinear tensor associated to a universally tensorial polynomial is of the form
$$\sum_{i,j,k}A_{ijk}\bullet_{\phi}\widetilde{\mathcal{T}}^{ijk}_{\phi}+\sum_{\varepsilon_i=\varepsilon_j=1}B_{ij}\bullet_{\phi}\widetilde{\mathcal{P}}^{ij}_{\phi}$$
for some $A_{ijk},B_{ij}\in \mathcal{R}$.

\section{Tensorial Ideals}\label{sec: tensorial}

In the last section, we showed that a polynomial $P$ is universally tensorial if and only if it vanishes on the set
$V_{\epsilon}$, by Theorem \ref{theo:unitens}. If we define the following ideals in $\mathcal{R}$, 
\begin{align*}
I_{\epsilon}^x&:=\langle y_i-\epsilon_i z_i:\ i=1\dots N\rangle,\nonumber\\
I_{\epsilon}^y&:=\langle  z_i-\epsilon_i x_i:\ i=1\dots N\rangle,\nonumber\\
I_{\epsilon}^z&:=\langle x_i-\epsilon_i y_i:\ i=1\dots N\rangle,\nonumber\end{align*}

\noindent then by the Real Nullstellensatz \cite[Theorem A.3]{Powers22} we have 
\[
\mathcal{I}_{\epsilon}=I(V_{\epsilon})=\sqrt[\mathbb R]{I_{\epsilon}^x\cap I_{\epsilon}^y\cap I_{\epsilon}^z}\,,
\]
where $\sqrt[\mathbb R]{J}$ denotes the real radical of an ideal $J$. It is easy to see that the $S_3$-action on the variables permutes these ideals and that each ideal is a prime ideal (indeed, $\mathcal{R}/I_\epsilon^x$ is isomorphic to a polynomial ring in fewer variables). First we will show that $I_{\epsilon}^x\cap I_{\epsilon}^y\cap I_{\epsilon}^z$ is a real radical ideal.

\begin{lem}
Let $I_{\epsilon}^x, I_{\epsilon}^y$ and $I_{\epsilon}^z$ be defined as above for some fixed integer $N>0$. Then
$\sqrt[\mathbb R]{I_{\epsilon}^x\cap I_{\epsilon}^y\cap I_{\epsilon}^z} =I_{\epsilon}^x\cap I_{\epsilon}^y\cap I_{\epsilon}^z$.   
\end{lem}
\begin{proof}
By \cite[Definition 3.2.13]{Sch}, since 
\[I_{\epsilon}^x\cap I_{\epsilon}^y\cap I_{\epsilon}^z\subseteq\sqrt[\mathbb{R}]{I_{\epsilon}^x\cap I_{\epsilon}^y\cap I_{\epsilon}^z} = \displaystyle\bigcap_{\substack{{I_{\epsilon}^x\cap I_{\epsilon}^y\cap I_{\epsilon}^z\subseteq J} \\ \text{$J$ is prime and real}}} \hspace{-6mm} J,\]
it suffices to show that $I_{\epsilon}^x$, $I_{\epsilon}^y$ and $I_{\epsilon}^z$ are real ideals (we showed that they are prime above). By \cite[Definition 3.2.13]{Sch}, a prime ideal $J\subset R$ is a real ideal if the field of fractions of $R/J$ is a real field (i.e. the residue field of $J$ can be ordered). The quotient $\mathcal{R}/I_{\epsilon}^x$ is isomorphic to $\R[x_1,\ldots, x_N, z_1,\ldots, z_N]$ whose field of fractions is the field of rational functions (which is known to be an ordered field). A similar statement holds for $I_{\epsilon}^y$ and $I_{\epsilon}^z$. Therefore, $\sqrt[\mathbb{R}]{I_{\epsilon}^x\cap I_{\epsilon}^y\cap I_{\epsilon}^z}$ is contained in $I_{\epsilon}^x,I_{\epsilon}^y$ and $I_{\epsilon}^z$, proving the result.

\end{proof}

We have reduced finding a generating set of $\mathcal{I}_{\epsilon}$ to finding a generating set for an intersection of ideals. In Gr\"obner theory, a common way to compute the intersection of two ideals $I,J\subset R=\mathbb{K}[u_1,\ldots,u_r]$ is to compute a Gr\"obner basis $\mathcal{G}$ for the ideal $T=tI + (1-t)J \subset R[t]$ with respect to an elimination order on $R[t]$. Since $I\cap J = T\cap R$, the elements in $\mathcal{G}$ that don't involve the variable $t$ generate $I\cap J$. See \cite[Section 15.10.4, Exercise 15.43]{E} for further details.

Using this method to compute $I_{\epsilon}^x\cap I_{\epsilon}^y\cap I_{\epsilon}^z$ directly for general $N$ is cumbersome. In particular, the process of finding a Gr\"obner basis for $J_{\epsilon}=tI_{\epsilon}^x + (1-t)I_{\epsilon}^y\cap I_{\epsilon}^z $ involves also computing a Gr\"obner basis for $I_{\epsilon}^x\cap I_{\epsilon}^y\cap I_{\epsilon}^z$, which is stronger than simply finding a generating set for it. Indeed, finding a conjectural Gr\"obner basis for $J_{\epsilon}$ or even $I_{\epsilon}^x\cap I_{\epsilon}^y\cap I_{\epsilon}^z$ was not immediately straightforward and the $S$-polynomial computations quickly became challenging. Instead, we will take advantage of the inductive definition of $I_{\epsilon}^x$, $I_{\epsilon}^y$ and $I_{\epsilon}^z$ to reduce the computation of a generating set to the $N=6$ case, and then check this case for all signatures by computer. 

We will proceed by introducing and subsequently applying some commutative algebra techniques that will hopefully prove useful in other geometric contexts. In Subsection \ref{subsec: Knutson}, we start  by describing the intersection of any two of the ideals $I_{\epsilon}^x, I_{\epsilon}^y$ and $I_{\epsilon}^z$ using a lesser known technique called Frobenius splitting \cite{Knutson}. We will then use Gr\"obner theory and the reduction described above to find a generating set for $\mathcal{I}_\epsilon$ in Subsection \ref{subsec: main}. We also provide a Gr\"obner basis for $I_{\epsilon}^x\cap I_{\epsilon}^y\cap I_{\epsilon}^z$ in the special case where $\epsilon=(1,1,\ldots,1)$ in Appendix \ref{app: alternate}, introducing tools that may be useful for subsequent research on this topic.

\subsection{Knutson ideals}\label{subsec: Knutson}

In this section we will prove that $I_\varepsilon^xI_\varepsilon^z = I_\varepsilon^x\cap I_\varepsilon^z$, which simplifies computations later. As highlighted in the last section, there are multiple general approaches to computing an intersection of two ideals, but because of the very specific linear form of the generators of each ideal, we will use the method called Frobenius splitting whose technical details, such as the issue of passing from positive characteristic to characteristic zero, can be conveniently described using Knutson ideals \cite{Seccia}.

\begin{mydef} \label{def: Knutson} Let $f\in R=\mathbb{K}[u_1,\ldots,u_r]$ be a polynomial such that $\init_<(f)$ is a squarefree monomial for some term order $<$ on $R$. Recall that given ideals $I,J$ in a commutative ring $A$, $(I:J)$ denotes the ideal of all $a\in A$ such  that $aJ\subseteq I$. Define $\mathcal{C}_f$ to be the smallest set of ideals satisfying the following conditions:

\begin{enumerate}
    \item $\langle f\rangle \in \mathcal{C}_f$,
    \item if $I\in \mathcal{C}_f$ then $(I:J)\in \mathcal{C}_f$ for every $J\subseteq R$,
    \item if $I$ and $J$ are in $\mathcal{C}_f$, then $I+J$ and $I\cap J$ are in $\mathcal{C}_f$.    
\end{enumerate}
If $I$ is an ideal in $\mathcal{C}_f$, we say that $I$ is a \textit{Knutson ideal associated to $f$}.
\end{mydef}

\begin{theorem}\label{thm: Knutson_radical} \cite[Main Theorem 1]{Seccia} Let $R = \mathbb{K}[u_1,\ldots,u_n]$ be a polynomial ring over any field $\mathbb{K}$ and let $f$ be a polynomial in $R$ such that $\init_<(f)= \prod_{i=1}^n u_i$ with respect to some term order $<$. If $I\in \mathcal{C}_f$, then $\init_<(I)$ is squarefree, and hence $I$ is radical.
\end{theorem}

Define the polynomial 
\[
F_{\varepsilon}^{xz} = \prod_{i=1}^N (x_i-\varepsilon_iy_i)(y_i-\varepsilon_iz_i)z_i \in\mathcal{R}\,.
\]

\noindent Under the monomial order $x_N>y_N>z_N>x_{N-1}>y_{N-1}>z_{N-1}>\cdots > x_1>y_1>z_1$, we see that 
\[\init_<(F_\varepsilon^{xz}) = \prod_{i=1}^Nx_iy_iz_i.\]

\begin{theorem}\label{thm: Knutson_product}
The ideals $I_\varepsilon^x$, $I_\varepsilon^z$ and $I_\varepsilon^xI_\varepsilon^z$ are Knutson ideals in $\mathcal{C}_{F_{\varepsilon}^{xz}}$.
\end{theorem}
\begin{proof}
Since $\init_<(F_\varepsilon^{xz})$ is squarefree, we may consider the set of Knutson ideals defined by $F_\varepsilon^{xz}$. Each factor of $F_\varepsilon^{xz}$ defines a principal Knutson ideal in $\mathcal{C}_{F_{\varepsilon}^{xz}}$ by property $(2)$ of Definition \ref{def: Knutson}. In particular, $\langle x_i-\varepsilon_iy_i \rangle$ and $\langle y_i-\epsilon_iz_i \rangle$ are Knutson ideals for $i=1,\ldots, N$. By property $(3)$ of Definition \ref{def: Knutson}, intersections and sums of Knutson ideals are also Knutson ideals. Therefore, $\langle x_i-\varepsilon_iy_i \rangle\cap \langle y_j-\varepsilon_jz_j \rangle = \langle (x_i-\varepsilon_iy_i)(y_j-\varepsilon_jz_j)\rangle$ is a Knutson ideal  for every pair $i,j$. This follows from the fact that the intersection of finitely many principal ideals in a unique factorization domain is again a principal ideal (generated by the least common multiple of the irreducible factors \cite[Lemma 13]{AldiDaSilvaGrandini24}). Since $I_\varepsilon^xI_\varepsilon^z$ can be written as a sum of Knutson ideals, namely \[I_\varepsilon^xI_\varepsilon^z=\displaystyle\sum_{i,j} \langle (x_i-\varepsilon_iy_i)(y_j-\varepsilon_jz_j)\rangle,\] then it is Knutson too. Similarly, $I_\varepsilon^x$ and $I_\varepsilon^z$ are Knutson since they are also a sum of Knutson ideals.
\end{proof}

\begin{cor}
$I_\varepsilon^xI_\varepsilon^z = I_\varepsilon^x\cap I_\varepsilon^z$, $I_\varepsilon^xI_\varepsilon^y = I_\varepsilon^x\cap I_\varepsilon^y$, and $I_\varepsilon^yI_\varepsilon^z = I_\varepsilon^y\cap I_\varepsilon^z$.
\end{cor}
\begin{proof} From the cyclic shift in variables (including in the monomial order) $x_i\rightarrow y_i \rightarrow z_i\rightarrow x_i$, it's sufficient to prove only the first identity. 

By Theorem \ref{thm: Knutson_radical}, since $\init_<(F_\varepsilon^{xz}) = \prod_{i=1}^Nx_iy_iz_i$, all of the Knutson ideals in $C_{F_\varepsilon^{xz}}$ are radical. Therefore, 
\[I_\varepsilon^xI_\varepsilon^z=\sqrt{I_\varepsilon^xI_\varepsilon^z} = \sqrt{I_\varepsilon^x\cap I_\varepsilon^z}=\sqrt{I_\varepsilon^x}\cap\sqrt{I_\varepsilon^z} = I_\varepsilon^x\cap I_\varepsilon^z\,.\]
\end{proof}

\begin{rem}\label{rem: Knutson Grobner}
By \cite{Knutson}, the natural generators of $I_\varepsilon^xI_\varepsilon^z$ are also a Gr\"obner bases for the ideal with respect to this monomial ordering, providing another proof of some of the statements in Proposition \ref{prop: GB_gens2} in the Appendix. In fact, the entire poset of Knutson ideals in $C_{F_\varepsilon^{xz}}$ is well-understood. 
\end{rem}

\subsection{Generating $\mathcal{I}_{\epsilon}$ }\label{subsec: main}

We are now prepared to describe a generating set for  $\mathcal{I}_{\epsilon}=I_{\epsilon}^x\cap I_{\epsilon}^y\cap I_{\epsilon}^z$ in terms of the polynomials $T^{ijk}_\varepsilon$, $P^{ij}_\varepsilon\in \mathcal I_\varepsilon$ introduced in, respectively, \eqref{eq:11} and \eqref{eq:14}. Computing a generating set for $I_{\epsilon}^x\cap I_{\epsilon}^y\cap I_{\epsilon}^z$ involves computing a Gr\"obner basis for \[J_\epsilon=tI_{\epsilon}^x + (1-t)I_{\epsilon}^y I_{\epsilon}^z \]

\noindent using the fact that $I_{\epsilon}^y I_{\epsilon}^z = I_{\epsilon}^y\cap I_{\epsilon}^z$. When computing a Gr\"obner basis for $J_\epsilon$, an elimination order is needed where the variable $t$ has the highest weight. For this section, we will use the lexicographic order $<_N$ defined by 
\[ t>x_N>y_N>z_N>x_{N-1}>y_{N-1}>z_{N-1}>\cdots > x_1>y_1>z_1.\]

In applying Buchberger's algorithm to the ideal $J_\epsilon$, we start with a generating set $\mathcal{G}$ of $J_\epsilon$, compute $S$-polynomials between elements of $\mathcal{G}$, reduce by $\mathcal{G}$ using the division algorithm to get a remainder, and adjoin the remainder to $\mathcal{G}$ if nonzero (see Section \ref{subsec: Grobner} for details). An important observation is that any element of $\mathcal{G}$ at any iteration of Buchberger's algorithm applied to $J_\epsilon$ with $<_N$ will only involve three indices. We will say that a polynomial $f$ is \emph{defined by at most three indices} if there exists  $i,j$ and $k$ such that $f\in \mathbb{R}[t,x_i,y_i,z_i,x_j,y_j,z_j,x_k,y_k,z_k]$.

The key to the argument that follows is to show that the division algorithm applied to any $S$-polynomial involving two generators defined by some fixed subset $\mathfrak{I}$ of indices $i,j,k,r,s,t$ will only involve that set of indices. Consider two polynomials $g^{ijk}$ and $g^{rst}$, each defined by at most three indices, with $S$-polynomial
\[S(g^{ijk},g^{rst}) = \frac{\init_{<_N}(g^{rst})}{\mathrm{GCD}(\init_{<_N}(g^{ijk}),\init_{<_N}(g^{rst}))}g^{ijk}- \frac{\init_{<_N}(g^{ijk})}{\mathrm{GCD}(\init_{<_N}(g^{ijk}),\init_{<_N}(g^{rst}))}g^{rst}.\]

\noindent We label the reduced polynomial on the right hand side by $f_{\mathfrak{I}}$. The division algorithm sequentially checks whether the lead term of $f_{\mathfrak{I}}$ is divisible by the lead term of each element in the current list of polynomials $\mathcal{G}$ from Buchberger's algorithm. We claim that any element of $\mathcal{G}$ whose lead term divides the lead term of $f_{\mathfrak{I}}$ must use only indices in $\mathfrak{I}$. A priori, it is possible that the division algorithm uses elements $g\in\mathcal{G}$ whose lead term has indices in $\mathfrak{I}$, but which involves other terms with indices not in $\mathfrak{I}$. We can show that there is an ordering of the generators in $\mathcal{G}$ where this does not happen. We will do this be converting the problem into one where $\mathfrak{I}\subseteq\{1,2,3,4,5,6\}$, as demonstrated in the next example.

\begin{example}\label{ex: index shift}
Given two generators of the ordered list $\mathcal{G}$ involving at most three indices, define an order preserving bijection between the indices $i,j,k,r,s,t$ (where some indices may be equal) with a subset of the indices $\{1,2,3,4,5,6\}$. 

For instance, if we are considering $S(T^{487}_{\varepsilon},T^{289}_{\varepsilon})$, then $\mathfrak{I}=\{2,4,7,8,9\}$, which is in bijection with $\{1,2,3,4,5\}$ via any permutation $\sigma\in S_9$ which sends $2$ to $1$, $4$ to $2$, $7$ to $3$, $8$ to $4$ and $9$ to $5$. We apply $\sigma$ to a polynomial by applying it to the indices of the variables in its support. Let $\varepsilon_\sigma$ be the signature of length five where entry $k$ in $\varepsilon_\sigma$ is equal to entry $\sigma^{-1}(k)$ in $\varepsilon$. For the $S$-polynomial, we get
\[ S(T^{487}_{\varepsilon},T^{289}_{\varepsilon}) = \sigma^{-1}(S(T^{243}_{\varepsilon_\sigma},T^{145}_{\varepsilon_\sigma})).  \]

\noindent If $\sum f_ig_j +r\in \mathbb{R}[x_1,\ldots,z_5]$ is the standard form for $S(T^{243}_{\varepsilon_\sigma},T^{145}_{\varepsilon_\sigma})$ with respect to $\sigma(\mathcal{G})$, then the hope is that 
\[\sigma^{-1}(\sum f_ig_j +r) =\sum \sigma^{-1}(f_i)\sigma^{-1}(g_j) +\sigma^{-1}(r)\]
is the standard form for $S(T^{487}_{\varepsilon},T^{289}_{\varepsilon})$ with respect to $\mathcal{G}$. There are several sufficient conditions to guarantee that we do get a standard form:

\begin{itemize}
    \item Taking initial terms and applying $\sigma$ should commute. That is, $\init_{<_5}(\sigma(f)) = \sigma(\init_{<_9}(f))$ so that applying $\sigma^{-1}$ results in a term which is actually a leading term.
    \item The indices appearing in the initial term of an element of $g\in \mathcal{G}$ should be the same as the indices appearing in $g$. This is to prevent issues with two elements having the same initial term but involve a different set of indices in their remaining terms. 
    \item Each element of $\mathcal{G}$ should involve at most three indices. This ensures that we only need to reduce to the $N=6$ case for each $S$-polynomial computation. 
    \item $\sigma^{-1}(g_i)$ must appear before $\sigma^{-1}(g_j)$ in the ordered list of generators $\mathcal{G}$ if $i<j$ and $g_i,g_j$ are in the ordered list for the current stage of Buchberger's algorithm for the $N=6$ case. Polynomial division depends on the ordering of the polynomials being divided (except when that list of polynomials is already a Gr\"obner basis).
\end{itemize}

\noindent Assuming that these conditions are met (and we will prove that this is enough below), then the standard form in the indices $\{1,2,3,4,5\}$ provides a standard form for the original $S$-polynomial in the indices $\{2,4,7,8,9\}$.

Therefore, if all $S$-polynomials reduce to $0$ for the indices $\{1,2,3,4,5,6\}$, then all $S$-polynomials reduce to $0$ for all arbitrary pairs of generators in general indices. Since we have verified the $N=6$ case by computer for all signatures, the result follows.  
\end{example}

\begin{rem}
It is worth noting that this procedure is reminiscent of a similar situation for Gr\"obner bases of patches of Hessenberg varieties. Similar technical details regarding $\init_<(\cdot)$ commuting with a ring isomorphism had to be checked in \cite{CDSHR}. In the previous example, the index bijection defines a ring isomorphism between $\mathbb{R}[\{x_i,y_i,z_i\}|i=2,4,7,8,9]$ and $\mathbb{R}[\{x_i,y_i,z_i\}|i=1,2,3,4,5]$, producing an easier mapping than the $\psi_w$ considered in \cite{CDSHR}.
\end{rem}

\begin{lem}\label{lem:  n<=6 case}
Fix a signature $\varepsilon$ and let $I_\varepsilon^x, I_\varepsilon^y,$ and $I_\varepsilon^z$ be ideals of $\mathcal{R}$ as defined above, endowed with the monomial order $<_N$. Assume that $N\leq 6$. Then 
\[\mathcal{G}_\epsilon= \{T^{ijk}_{\varepsilon}|1\leq i,j,k\leq N\} \cup \{P^{ij}_{\varepsilon}| \varepsilon_i=\varepsilon_j=1\}\]

\noindent is a generating set of $\mathcal{I}_{\varepsilon}=I_\varepsilon^x\cap I_\varepsilon^y\cap I_\varepsilon^z$. Additionally, for each element $g$ of the reduced Gr\"obner basis of $J_\epsilon^N=tI_{\epsilon}^x + (1-t)I_{\epsilon}^y I_{\epsilon}^z$ with respect to $<_N$, there exist at most three indices $\mathfrak{I}\subseteq \{i,j,k\}$ such that $g\in \mathbb{R}[t,x_r,y_r,z_r|r\in \mathfrak{I}]$. Furthermore, $\init_{<_N}(g)$ uses all of the indices of $\mathfrak{I}$.
\end{lem}
\begin{proof}
This can be checked directly using Macaulay2, for example by using the code provided in the ancillary files. 

Note that the ideal $J_\epsilon^N$ can be generated by $\{t(y_i-\varepsilon_iz_i),(1-t)(z_i-\varepsilon_ix_i)(x_j-\varepsilon_jy_j)\}_{i,j}$ where $1\leq i,j\leq N$, and we choose to  order the generators so that the elements $t(y_i-\varepsilon_iz_i)$ appear first in increasing order with respect to the index $i$, followed by the elements $(1-t)(z_i-\varepsilon_ix_i)(x_j-\varepsilon_jy_j)$, where $(i,j)$ appears before $(i',j')$ if $i<i'$ or $i=i'$ and $j<j'$.
\end{proof}

\begin{theorem}\label{thm: reduce to n=6}
Fix $N\in\mathbb{Z}_{>0}$ and a signature $\varepsilon$. Let $I_\varepsilon^x, I_\varepsilon^y,$ and $I_\varepsilon^z$ be ideals of $\mathcal{R}$ as defined above, endowed with the monomial order $<_N$. Then
\[\mathcal{G}_\epsilon= \{T^{ijk}_{\varepsilon}|1\leq i,j,k\leq N\} \cup \{P^{ij}_{\varepsilon}| \varepsilon_i=\varepsilon_j=1\}\]

\noindent is a generating set of $\mathcal{I}_{\varepsilon}=I_\varepsilon^x\cap I_\varepsilon^y\cap I_\varepsilon^z$.
\end{theorem}

\begin{proof}
The ideal $J_\epsilon^N=tI_{\epsilon}^x + (1-t)I_{\epsilon}^y I_{\epsilon}^z$ is generated by the polynomials $\{t(y_i-\varepsilon_iz_i),(1-t)(z_i-\varepsilon_ix_i)(x_j-\varepsilon_jy_j)\}_{i,j}$ where $1\leq i,j\leq N$. Denote this generating set by $\mathcal{G}_N^1(\varepsilon) = \{g_1,\ldots,g_{n_1}\}$, ordered as in the proof of Lemma \ref{lem:  n<=6 case}. Consider Buchberger's algorithm applied to $\mathcal{G}_N^1(\varepsilon)$. Any $S$-polynomial between $g_i$ and $g_j$ will involve no more than four indices (in fact, no more than six in the general inductive step). There is an order preserving bijection $\sigma$ between these $r$ indices and the set $\{1,\ldots, r\}$. As in Example \ref{ex: index shift}, let $\varepsilon_\sigma$ be a signature defined from $\varepsilon$ and $\sigma$, and let $\sigma(g_i)= \bar{g_i}$ be defined by applying $\sigma$ to the subscript/index of each indeterminate. Then
\[S(g_i,g_j) = \sigma^{-1}(\bar{S}(\bar{g_i},\bar{g_j}))\]

\noindent where $\bar{S}(\bar{g_i},\bar{g_j}) \in \mathcal{R}_6=\mathbb{R}[x_1,y_1,z_1,\ldots, x_6,y_6,z_6]$. Performing polynomial division with respect to $\mathcal{G}^1_6(\varepsilon_\sigma)$ allows us to write
\[\bar{S}(\bar{g_i},\bar{g_j}) = \sum_{\alpha} f_\alpha h_\alpha +r_{i,j}\]

\noindent where $h_\alpha \in \mathcal{G}^1_6(\varepsilon_\sigma)$ and $f_\alpha,r_{i,j}\in \mathcal{R}_6$. By definition of $J_\epsilon^N$, $\sigma^{-1}(h_\alpha)\in \mathcal{G}_N^1(\varepsilon)$. Then \[S(g_i,g_j) = \sum_{\alpha} \sigma^{-1}(f_\alpha)\sigma^{-1}(h_\alpha) + \sigma^{-1}(r_{i,j}).\]

We claim that this is the standard form for $S(g_i,g_j)$ with respect to $\mathcal{G}_N^1(\varepsilon)$. By definition of $<_N$, we have $\init_{<_N}(\sigma^{-1}(f)) = \sigma^{-1}(\init_{<_6}(f))$ for any $f\in \mathcal{R}_6$. Assume that for some $k$, $\init_{<_N}(g_k)$ divides some term of $\sigma^{-1}(r_{i,j})$. This would imply that $\init_{<_6}(\bar{g_k})$ divides some term of $r_{i,j}$ since $\init_{<_N}(g_k)=\init_{<_N}(\sigma^{-1}(\bar{g_k})) = \sigma^{-1}(\init_{<_6}(\bar{g_k}))$, a contradiction. So $\sigma^{-1}(r_{i,j})$ satisfies the conditions of a remainder term. 

Next, assume that $g_{A}$ is the first element in the ordered list $\mathcal{G}_N^1(\varepsilon)$ such that $\init_{<_N}(g_{A})$ divides the initial term of $S(g_i,g_j)$, but that $\sigma(g_A)$ is not the first element of $\mathcal{G}^1_6(\varepsilon_\sigma)$ whose initial term divides $\bar{S}(\bar{g_i},\bar{g_j})$. Then there is some element $g_B$ appearing before $g_A$ in $\mathcal{G}^1_6(\varepsilon_\sigma)$ where $\init_{<_6}(g_B)$ divides $\init_{<_6}(\bar{S}(\bar{g_i},\bar{g_j}))$. This contradicts the consistent ordering of $\mathcal{G}_N^1(\varepsilon)$ described in Lemma \ref{lem:  n<=6 case}. In particular, if $g_A$ uses one index, then $g_B$ uses one index, and $\sigma$ preserves the index ordering, so $g_B$ appearing before $g_A$ in $\mathcal{G}^1_6(\varepsilon_\sigma)$ implies that $\sigma^{-1}(g_B)$ appears before $\sigma^{-1}(g_A)$ in $\mathcal{G}_N^1(\varepsilon)$. A similar statement holds for more indices. 

Finally, assume that in performing the polynomial division on $S(g_i,g_j)$, there is some element $g_C$ such that $\init_{<_N}(g_C)$ divides the current polynomial, but $g_C$ uses indices other than the indices appearing in the original $S(g_i,g_j)$. In this case, it may be possible that new indices appear in the division algorithm which don't match the indices $\sigma^{-1}(\{1,2,3,4,5,6\})$ appearing in $\sigma^{-1}(\bar{S}(\bar{g_i},\bar{g_j}))$. Since $g_C$ uses at most 3 indices, there is some order preserving bijection $\pi$ from these indices to a subset of $\{1,2,3\}$. By Lemma \ref{lem:  n<=6 case}, $\pi(g_C)$ has the property that the indices of $g_C$ are the same as the indices of its initial term under $<_3$. Since $\init_{<_N}(g_C)$ divides a term with indices in $\sigma^{-1}(\{1,2,3,4,5,6\})$, all of $g_C$ must be defined in these indices too, a contradiction.

This proves that the multivariate polynomial division algorithm applied to $\mathcal{G}_N^1(\varepsilon)$ with respect to $<_N$ produces a standard form which is in bijection through $\sigma$ with the standard form coming from $\mathcal{G}_6^1(\varepsilon_\sigma)$.

To complete the first step of Buchberger's algorithm, perform this polynomial division for all pairs $i,j$ with $i<j$. Create an extended list of polynomials $\mathcal{G}_N^2(\varepsilon)$ by adjoining to the list $\mathcal{G}_N^1(\varepsilon)$ any non-zero remainder term $\sigma^{-1}(r_{i,j})$ (here $\sigma$ actually varies based on the indices, but we will abuse notation and use $\sigma$ to represent the general situation). This list should be ordered so that elements with fewer indices appear first, and $(i,j,k)$ appears before $(i',j',k')$ if $i<i'$, or $j<j'$ if $i=i'$, or $k<k'$ if both $i=i'$ and $j=j'$. In the case that there is a set of two or more elements that use the same set of indices $\{i,j,k\}$ (but possibly in different sets of variables), then order this set so that $f$ appears before $h$ if $\sigma(f)$ appears before $\sigma(h)$ in $\mathcal{G}_6^2(\varepsilon_\sigma)$. This ensures consistency with the list being produced in the $N=6$ code computation.  

By Lemma \ref{lem:  n<=6 case}, $r_{i,j}$ is defined by at most three indices, and therefore $\sigma^{-1}(r_{i,j})$ is too. Then $\mathcal{G}_N^2(\varepsilon)$ consists of polynomials which are each defined by at most three indices. This is important since all $S$-pairs will involve at most six indices, allowing the computation to be reduced to the $N=6$ case. Continuing the process, we can successively compute each $\mathcal{G}^\ell_N(\varepsilon)$, until all remainders $r_{i,j}$ are $0$ (which is guaranteed by Buchberger's algorithm). The terminal list $\mathcal{G}_N^L(\varepsilon)$ will be a Gr\"obner basis for $J_\epsilon^N$ with respect to $<_N$. The subset of elements $\mathcal{H}$ in this list which do not involve $t$ will be a Gr\"obner basis for $\mathcal{I}_{\varepsilon}$ (that is $\mathcal{H} = \mathcal{G}_N^L(\varepsilon)\cap \mathcal{R}$). 

The result now follows by Lemma \ref{lem:  n<=6 case} since for any $h\in \mathcal{H}$, $\sigma(h)$ can be written as a combination of $T^{ijk}_{\varepsilon_\sigma}$ and $P^{ij}_{\varepsilon_\sigma}$ for $N=6$, which means that $h$ can be written as a combination of $\sigma^{-1}(T^{ijk}_{\varepsilon_\sigma}) = T^{\sigma^{-1}(i)\sigma^{-1}(j)\sigma^{-1}(k)}_{\varepsilon}$ and $\sigma^{-1}(P^{ij}_{\varepsilon_\sigma})=P^{\sigma^{-1}(i)\sigma^{-1}(j)}_{\varepsilon}$, proving that the set $\mathcal{G}_\varepsilon$ is a generating set of $\mathcal{I}_{\varepsilon}$. 
\end{proof}

\begin{rem}
The proof of Theorem \ref{thm: reduce to n=6} shows that the Gr\"obner basis computation of $J_\epsilon^N$ can be reduced to a series of independent Gr\"obner basis computations for a fixed set of at most six indices $\mathcal{J}$ and fixed signature $\varepsilon_\sigma$. The entire argument works because for a fixed $\mathcal{J}$, the $S$-pair computations occurring between the polynomials of $\mathcal{G}^\ell_N$ with indices in $\mathcal{J}$ are in bijection with the $S$-pair computations occurring in $\mathcal{G}_6^\ell(\varepsilon_\sigma)$. This reduces the entire Gr\"obner basis computation to repeated independent copies of the Gr\"obner basis computation of $J_\epsilon^6$ with varying signatures $\varepsilon$. Therefore, Theorem \ref{thm: reduce to n=6} is a careful bookkeeping argument to show this reduction. 
\end{rem}

As a final note, for the same reasons as described in the previous remark, finding a Gr\"obner basis of $\mathcal{I}_{\varepsilon}$ can be reduced to finding a Gr\"obner basis for the $N=6$ case. We investigate this question in Appendix \ref{app: alternate}. Although we did not need an explicit Gr\"obner basis to provide a generating set of $\mathcal{I}_{\varepsilon}$, we did briefly explore this question, and with respect to $<_N$ there were additional polynomials besides $T^{ijk}_{\varepsilon}$ and $P^{ij}_{\varepsilon}$ that were needed in the Gr\"obner basis. For the signature with all $-1$'s, these elements could be viewed as the determinant of a certain $3\times 3$ matrix. Unfortunately we could not give a general description for other signatures, although the extra elements appeared to be determinantal in nature. This could be the topic of future research, especially if Gr\"obner degeneration techniques prove useful in studying tensorial constraints in generalized geometry. 

\section{Uniform Integrability}\label{sec:5}
\noindent Given a commuting family $\phi$, we can now consider the set $\mathcal{Z}(\phi)$ of all $P\in \mathcal{R}$ such that $P\bullet_{\phi}\tau_C$ vanishes identically.
It is straightforward to verify that $\mathcal{Z}(\phi)$ is an ideal of $\mathcal{R}$ contained in $\mathcal{I}(\phi)$. We will say that $\phi$ is \emph{uniformly integrable} if $\mathcal{I}_{\varepsilon}\subseteq \mathcal{Z}(\phi)$.

\subsection{Semisimple families}
Following \cite{AldiGrandini22}, we say that an endomorphism $\phi$ of the generalized tangent bundle is {\it semisimple} if there is a finite set $\Sigma$ of complex numbers and a decomposition $\T M\otimes \mathbb C\cong \bigoplus_{\lambda\in \Sigma}L_\lambda$ into eigenbundles $L_\lambda=\ker(\phi-\lambda I)$ of (the complexification of) $\phi$. We will call a commuting family $\phi=(\phi_1,\ldots,\phi_N)$ \emph{semisimple} if each endomorphism $\phi_i$ is semisimple. Note that for every semisimple commuting family, the complexified generalized tangent bundle $\T M\otimes \mathbb C$ decomposes into the direct sum of subbundles of the form
\[
L_{\mathbf \lambda}=\bigcap_{i=1}^N \ker(\phi_i-\lambda_iI)\,,
\]
labeled by vectors $\mathbf \lambda=(\lambda_1,\ldots,\lambda_n)$. 
The vectors $\lambda$ will be called the eigenvalues of the family $\phi$. It is straightforward to prove that $\T M\otimes \C$ splits as a direct sum of the subbundles $L_{\lambda}$.

\begin{example}
Let $N=1$ so that $\phi=(\phi_1)$ and assume $\phi_1$ is skew-symmetric and semisimple. Then $\phi$ is uniformly integrable if and only if $\mathcal S_\phi=0$ if and only if
\begin{equation}\label{eq:14}
(\lambda_1+\mu_1)\lbra L_\lambda,L_\mu\rbra\subseteq L_\lambda+L_\mu\,,
\end{equation} 
where the last equivalence follows from \cite{AldiGrandini22}. In particular, a generalized almost complex structure is uniformly integrable if and only if it is a generalized complex structure if and only if its two eigenbundles $L_{\pm \sqrt{-1}}$ are both involutive with respect to the Courant-Dorfman bracket.
\end{example}

The following Theorem generalizes the equivalent condition \eqref{eq:14} to arbitrary commuting families of semisimple endomorphisms of the generalized tangent bundle.

\begin{theorem}\label{prop:ssUnifInt}
A semisimple commuting family $\phi$ is uniformly integrable if and only if for all $\lambda,\mu\in \C^N$, 
\begin{enumerate}[1)]
\item 
$\displaystyle\|\lambda-{\rm diag}(\varepsilon)(\mu)\|\cdot\lbra\Gamma(L_{\lambda}),\Gamma(L_{\mu})\rbra\subseteq \Gamma(L_{\lambda} + L_{\mu}),$ and, 
\item $\displaystyle\lbra\Gamma(L_{\lambda}),\Gamma(L_{\mu})\rbra\subseteq \Gamma\left(L_{\lambda}\circledast L_{\mu}\right)$ 
\end{enumerate}
where $L_{\lambda}\circledast L_{\mu}$ denotes the Whitney sum of all eigenbundles $L_{\xi}$ such that
$$\det\left(\begin{array}{ccc}\lambda_i&\lambda_j&1\\
     \mu_i&\mu_j&1\\
     \xi_i&\xi_j&1\end{array}\right)=0$$
      for all indices $i,j$ such that $\epsilon_i=1=\epsilon_j$.

\end{theorem}
\begin{proof}
If $\phi$ is uniformly integrable, then by Theorem \ref{thm: reduce to n=6} all torsions $\mathcal T_\phi^{ijk}$ vanish. In particular, by \cite{AldiGrandini22}, the vanishing of each $\mathcal T_\phi^{iii}=-\mathcal S_{\phi_i}$ implies that
\begin{equation}
(\lambda_i-\varepsilon_i\mu_i)\lbra\Gamma(L_\lambda),\Gamma(L_\mu)\rbra\subseteq L_\lambda+L_\mu
\end{equation}
for every $i\in \{1,\ldots,N\}$. This proves 1). On the other hand, if $\a\in \Gamma(L_\lambda)$, $\b\in \Gamma(L_\mu)$, then
\begin{equation}
\mathcal K_{(\phi_k,\phi_j)}(\a,\b)=(\lambda_k-\phi_k)(\lambda_j-\phi_j)\lbra \a,\b\rbra\,.
\end{equation}
Substituting this into \eqref{eq:10} we obtain that
\begin{align*}
\mathcal T_\phi^{ijk}(\a,\b) &= -\varepsilon_k(\lambda_i-\varepsilon_i\mu_i)\mathcal K_{(\phi_k,\phi_j)}(\a,\b)\\
&=-\varepsilon_k(\lambda_i-\varepsilon_i\mu_i)(\lambda_k-\phi_k)(\lambda_j-\phi_j)\lbra \a,\b\rbra
\end{align*} 
vanishes if 1) holds. Hence, 1) is equivalent to the vanishing of all $\mathcal T^{ijk}_\phi$. Similarly, whenever $i,j$ are such that $\varepsilon_i=1=\varepsilon_j$,
\begin{align*}
\mathcal P^{ij}_\phi(\a,\b) & = \mathcal K_{(\phi_i,\phi_j)}(\a,\b)-\mathcal K_{(\phi_j,\phi_i)}(\a,\b)\\
&=((\lambda_i-\phi_i)(\mu_j-\phi_j)-(\lambda_j-\phi_j)(\mu_i-\phi_i))\lbra \a,\b\rbra\,,
\end{align*}
which is readily seen to vanish if and only if 2) holds.

\end{proof}
\begin{rem} 
An alternate proof of Theorem \ref{prop:ssUnifInt} can be obtained using the following argument. If $\phi$ is a semisimple commuting family and $P\in \mathcal{I}(\phi)$, then $P\in \mathcal{Z}(\phi)$ if and only if for all triples of eigenvalues $\lambda,\mu,\xi$ we have
$$P(\lambda,\mu,\xi)\langle\lbra\Gamma(L_{\lambda}),\Gamma(L_{\mu})\rbra, \Gamma(L_{\xi})\rangle=0.$$
This is equivalent to saying that $P\in \mathcal{Z}(\phi)$ if and only if for all pairs of eigenvalues $\lambda,\mu$ we have
$$\lbra\Gamma(L_{\lambda}),\Gamma(L_{\mu})\rbra\subseteq \Gamma\left(\bigcap_{P(\lambda,\mu,\xi)\neq 0}L_{\xi}^{\bot}\right)=\Gamma\left(\left(\bigoplus_{P(\lambda,\mu,\xi)\neq 0}L_{\xi}\right)^{\bot}\right)=\Gamma\left(\bigoplus_{P(\lambda,\mu,\xi)= 0}L_{{\rm diag}(\varepsilon)\xi}\right).$$
More generally, a subset $S\subseteq \mathcal{I}(\phi)$ is also included in $\mathcal{Z}(\phi)$ if and only if, for all pairs of eigenvalues $\lambda,\mu$, the bracket $\lbra\Gamma(L_{\lambda}),\Gamma(L_{\mu})\rbra $ is contained in the space of sections of the Whitney sum of all eigenbundles $L_{{\rm diag}(\epsilon)\xi}$ such that $P(\lambda,\mu,\xi)=0$ for all $P\in S$. 
\end{rem}

If $\phi=(\phi_1)$ with $\phi_1$ symmetric, then we know from Example \ref{ex:18} that $\phi$ is uniformly integrable if and only if the shifted Courant-Nijenhuis torsion of $\phi_1$ vanishes. Under the additional assumption that $\phi_1$ is semisimple, this is reflected in condition 2) of Theorem \ref{prop:ssUnifInt} being trivially satisfied. Hence, in this case uniform integrability is equivalent to
\begin{equation}\label{eq:17}
(\lambda_1-\mu_1)\lbra L_\lambda,L_\mu\rbra \subseteq L_\lambda+L_\mu\,.
\end{equation}
It follows from \eqref{eq:17} that any symmmetric endomorphism of the generalized tangent bundle such that $\phi_1^2=w$ for some non-negative scalar $w$ is uniformly integrable. We now explore some special cases in the next series of examples.

\begin{example}
Let $\phi=(\phi_1,\phi_2)$ be such that $\phi_1$ and $\phi_2$ are generalized almost complex structures. We know from Example \ref{ex:16} that $\phi$ is uniformly integrable if and only if $\phi_1$ and $\phi_2$ are generalized complex structures. Since $\phi$ is semisimple and
\begin{equation}
\mathbb T M\otimes\mathbb C\cong L_{\sqrt{-1},\sqrt{-1}}\oplus L_{\sqrt{-1},-\sqrt{-1}}\oplus L_{-\sqrt{-1},\sqrt{-1}}\oplus L_{-\sqrt{-1},-\sqrt{-1}}\,,
\end{equation}
Theorem \ref{prop:ssUnifInt} recovers the well known observation that a commuting pair of generalized almost complex structures is integrable if and only if the $\pm \sqrt{-1}$-eigenbundles of both $\phi_1$ and $\phi_2$ are involutive with respect to the Courant-Dorfman bracket \cite{Gualtieri14}. 

Equivalently, a (not necessarily positive definite) generalized K\"ahler structure can be encoded as a commuting pair $(\phi_1,\phi_2)$ where $\phi_1$ is generalized complex and  $\phi_2$ is a generalized Riemannian metric. The corresponding decomposition of the complexified generalized tangent bundle becomes
\begin{equation}
\mathbb T M\otimes\mathbb C\cong L_{\sqrt{-1},1}\oplus L_{\sqrt{-1},-1}\oplus L_{-\sqrt{-1},1}\oplus L_{-\sqrt{-1},-1}\,.
\end{equation}
By Theorem  \ref{prop:ssUnifInt} the corresponding involutivity conditions are
\begin{equation}\label{eq:16}
\lbra L_{\sqrt{-1},\pm 1}, L_{\sqrt{-1},\pm 1} \rbra\subseteq L_{\sqrt{-1},\pm 1}\,.
\end{equation}
\end{example}

\begin{example}
More generally, let $\phi=(\phi_1,\phi_2)$ be a metric $F$-structure in the sense of Vaisman \cite{Vaisman08} so that in particular $\phi_1^3+\phi_1=0$ and $\mathcal \phi_2^2=1$. Then $\phi$ is generalized CRFK \cite{Vaisman08} if and only if \eqref{eq:16} holds and 
\begin{equation}\label{eq:23}
\lbra L_{\sqrt{-1},\pm 1},L_{0,\pm 1}\rbra \subseteq L_{\sqrt{-1},\pm 1}\oplus L_{0,\pm 1}\,.
\end{equation}
By Theorem \ref{prop:ssUnifInt}, \eqref{eq:23} holds if and only if $\phi$ is uniformly integrable.
\end{example}

\subsection{Commuting pairs of generalized metrics}
Thus far, we have studied the uniform integrability condition in the case of semisimple families. We will now specialize the discussion to pairs of metrics. We will first start with an example. 
\begin{example}
Let $\phi=(\phi_1,\phi_2)$ be a commuting pair of generalized metric structures. Then
\begin{equation}
\mathbb T M\otimes\mathbb C\cong L_{1,1}\oplus L_{1,-1}\oplus L_{-1,1}\oplus L_{-1,-1}
\end{equation}
and, by Proposition \eqref{prop:ssUnifInt}, $\phi$ is uniformly integrable if and only if 
\begin{equation}
\lbra L_{1,\lambda_2},L_{-1,\mu_2}\rbra \subseteq L_{1,\lambda_2}\oplus L_{-1,\mu_2}
\end{equation}
and 
\begin{equation}
\lbra L_{\lambda_1, 1},L_{\mu_1,-1 }\rbra \subseteq L_{\lambda_1, 1}\oplus L_{\mu_1,-1}\,.
\end{equation}
for any $\lambda_1,\lambda_2,\mu_1,\mu_2\in \{-1,1\}$
\end{example}

\begin{prop}\label{prop:34} Let $\phi_1,\phi_2$ be commuting, positive definite generalized Riemannian metrics. Then $\phi_1=\phi_2$.
\end{prop}
\begin{proof} Any positive definite generalized Riemannian metric can be written \cite{Gualtieri11} in the form $e^{\omega}\psi e^{-\omega}$, for some classical, positive definite Riemannian metric $\psi$ and $B$-field $\omega$. Then, up to an overall $B$-field transformation, we can assume $\phi_2$ to be classical. Hence we have the matrix representation (with respect to the usual splitting $\T M=TM\oplus T^*M$)
$$\phi_1=\left[\begin{array}{cc}-g_1^{-1}\omega&g_1^{-1}\\g_1-\omega g_1^{-1}\omega&\omega g_1^{-1}\end{array}\right], \qquad \phi_2=\left[\begin{array}{cc}0&g_2^{-1}\\g_2&0\end{array}\right]$$
where $g_1,g_2$ are classical, positive definite Riemannian metrics and $\omega$ is a two-form, viewed as maps $g_i,\omega:TM\rightarrow T^*M$. Then, $\phi_1\phi_2=\phi_2\phi_1$ if and only if
$$\left[\begin{array}{cc}g_1^{-1}g_2&-g_1^{-1}\omega g_2^{-1}\\
\omega g_1^{-1}g_2&g_1g_2^{-1}-\omega g_1^{-1}\omega g_2^{-1}\end{array}\right]=\left[\begin{array}{cc}g_2^{-1}g_1-g_2^{-1}\omega g_1^{-1}\omega&g_2^{-1}\omega g_1^{-1}\\-g_2g_1^{-1}\omega&g_2g_1^{-1}\end{array}\right]$$
and from the $(2,2)$ entries and the $(1,2)$ entries left-multiplied by $\omega$ we obtain the conditions
\begin{eqnarray}\label{eq: comm}g_1g_2^{-1}&=&g_2g_1^{-1}+\omega g_1^{-1}\omega g_2^{-1},\nonumber\\\\
\omega g_1^{-1}\omega g_2^{-1}&=&-\omega g_2^{-1}\omega g_1^{-1}.\nonumber\end{eqnarray}
Now, fix a coordinate chart, and let $A_1,A_2,B$ be the local matrix representations of $g_1, g_2, \omega$ respectively. Then, the equations (\ref{eq: comm})  are expressed locally as
\begin{equation}\label{eq:28}
A_1A_2^{-1}=A_2A_1^{-1}+BA_1^{-1}BA_2^{-1},\qquad BA_1^{-1}BA_2^{-1}+BA_2^{-1}BA_1^{-1}=0\,.
\end{equation}
Since $A_1, A_2$ are symmetric and positive definite, we can consider their square roots. Multiplying the equations \eqref{eq:28} on the left by $A_2^{-1/2}$  and on the right by $A_2^{1/2}$ yields the equations
$$\left(A_2^{-1/2}A_1^{1/2}\right)\left(A_1^{1/2}A_2^{-1/2}\right)=\left(A_2^{1/2}A_1^{-1/2}\right)\left(A_1^{-1/2}A_2^{1/2}\right)+\left(A_2^{-1/2}BA_1^{-1/2}\right)\left(A_1^{-1/2}BA_2^{-1/2}\right)$$
and
$$\left(A_2^{-1/2}BA_1^{-1/2}\right)\left(A_1^{-1/2}BA_2^{-1/2}\right)+\left(A_2^{-1/2}A_1^{1/2}\right)\left(A_1^{-1/2}BA_2^{-1/2}\right)\left(A_2^{-1/2}BA_1^{-1/2}\right)\left(A_1^{-1/2}A_2^{1/2}\right)=0$$
which can be expressed in compact form as
$$RR^{T}=(RR^T)^{-1}-SS^T, \qquad SS^T+RS^TSR^{-1}=0$$
where
$$R=A_2^{-1/2}A_1^{1/2}, \qquad S=A_2^{-1/2}BA_{1}^{-1/2}\,.$$

Now, taking the trace of the second equation yields ${\rm tr}(SS^T)=0$.  However, since $SS^T$ must be positive semidefinite we conclude $SS^T=0$, i.e. $S=0$, which in turn implies $\omega=0$. Moreover, $(RR^T)^2=1$, which implies $RR^T=1$ since $RR^T$ is positive definite. This in turn implies $A_1=A_2$, i.e.\ $g_1=g_2$.
\end{proof}

Let us now consider the case of a commuting pair $\phi$ when $\phi_1,\phi_2$ are both classical, not necessarily positive definite, metrics. Using the notation above, we have $\omega=0$, and
$$g_1^{-1}g_2=g_2^{-1}g_1:=T$$
where $T:TM\rightarrow TM$ is an involution such that
$$g_1(TX,TY)=g_1(X,Y)$$
for all $X,Y\in \Gamma(TM)$. Therefore, the data $(\phi_1, \phi_2)$ is equivalent to $(g_1,T)$, a \emph{metric almost product structure} or, equivalently, a splitting (orthogonal with respect to $g_1$)
$$TM=E_+\oplus E_-$$
where $E_{\pm}$ is the $\pm 1$-eigenbundle of $T$.

\begin{prop} Let $\phi=(\phi_1, \phi_2)$ and $g_1$ be as in Proposition \ref{prop:34}. Then $\phi$ is uniformly integrable if and only if the following conditions are satisfied:
\begin{enumerate}[i)]
\item $[\Gamma(E_+),\Gamma(E_+)]\subseteq\Gamma(E_+)$ and $[\Gamma(E_-),\Gamma(E_-)]\subseteq\Gamma(E_-)$, and,
\item for all $X\in \Gamma(E_+)$ and all $Y\in \Gamma(E_-)$ we have
$$\left.\left(\mathcal{L}_{X}g_1\right)\right|_{E_-\otimes E_-}=0\qquad \mbox{ and } \qquad\left.\left(\mathcal{L}_{Y}g_1\right)\right|_{E_+\otimes E_+}=0.$$
\end{enumerate}
\end{prop}

\begin{proof}
The four eigenbundles of $\phi$ are
$$L_{++}=\{X+g_1(X)\,|\,X\in E_+\}, \qquad L_{+-}=\{X+g_1(X)\,|\,X\in E_-\},$$
$$L_{-+}=\{X-g_1(X)\,|\,X\in E_-\}, \qquad L_{--}=\{X-g_1(X)\,|\,X\in E_+\}$$
so that $\T M=L_{++}\oplus L_{+-}\oplus L_{-+}\oplus L_{--}$. For example, $X\in E_-$ implies $g_1(X)=-g_2(X)$ and thus
$$\phi_1(X-g_1(X))=g_1(X)-g_1^{-1}g_1(X)=g_1(X)-X=-(X-g_1(X))$$
and 
$$\phi_2(X-g_1(X))=g_2(X)-g_2^{-1}g_1(X)=g_2(X)+X=X-g_1(X)$$
so that $X-g_1(X)\in L_{-+}$. Note that $\sigma(L_{++})=L_{--}$ and $\sigma(L_{+-})=L_{-+}$ where $\sigma(X+\alpha):=X-\alpha$. Since $\sigma$ is an involution and a Courant anti-isometry, the pair $\phi$ is uniformly integrable if and only if the following r conditions are satisfied:
\begin{enumerate}[1)]
\item $\lbra L_{++},L_{--}\rbra\subseteq L_{++}\oplus L_{--}=E_+\oplus g_1(E_+)$;
\item $\lbra L_{+-},L_{-+}\rbra\subseteq L_{+-}\oplus L_{-+}=E_-\oplus g_1(E_-)$;
\item $\lbra L_{++},L_{+-}\rbra\subseteq L_{++}\oplus L_{+-}=\Gamma_{g_1}$;
\item $\lbra L_{++},L_{-+}\rbra\subseteq L_{++}\oplus L_{-+}=\Gamma_{g_2}$;
\end{enumerate}
where $\Gamma_g$ denotes the graph of $g:TM\rightarrow T^*M$. We analyze these conditions separately:
\begin{enumerate}[1)]
\item For all $X,Y\in \Gamma(E_+)$, the generalized vector
$$\lbra X+g_1(X),Y-g_1(Y)\rbra=[X,Y]-i_Xd(g_1(Y))-i_Yd(g_1(X))-d(g_1(X,Y))$$
is a section of $E_+\oplus g_1(E_+)$ if and only if
$$[X,Y]\in \Gamma(E_+)$$
and
$$-i_Zi_Xd(g_1(Y))-i_Zi_Yd(g_1(X))-i_Zd(g_1(X,Y))=0
\mbox{ for all } Z\in \Gamma(E_-).$$
This second condition
is equivalent to
$$g_1(Y,[X,Z])+g_1(X,[Y,Z])+Zg_1(X,Y)=0,$$
that is, $\left({\mathcal L}_Zg_1\right)(X,Y)=0$;
\item Similarly, for all $X,Y\in \Gamma(E_-)$, the generalized vector
$$\lbra X+g_1(X),Y-g_1(Y)\rbra=[X,Y]-i_Xd(g_1(Y))-i_Yd(g_1(X))-d(g_1(X,Y))$$
is a section of $E_-\oplus g_1(E_-)$ if and only if
$$[X,Y]\in \Gamma(E_-)$$
and
$$i_Zi_Xd(g_1(Y))+i_Zi_Yd(g_1(X))+i_Zd(g_1(X,Y))=0
\mbox{ for all } Z\in \Gamma(E_+).$$
This second condition 
is equivalent to
$$-Zg_1(X,Y)-g_1(Y,[X,Z])-g_1(X,[Y,Z])=0,$$
that is, $\left({\mathcal L}_Zg_1\right)(X,Y)=0$;
\item Now, for all $X\in \Gamma(E_+)$ and $Y\in \Gamma(E_-)$ the generalized vector
$$\lbra X+g_1(X),Y+g_1(Y)\rbra=[X,Y]+i_Xd(g_1(Y))-i_Yd(g_1(X))$$
is a section of $\Gamma_{g_1}$ if and only if, for all $Z\in \Gamma(E_+)$ and all $W\in \Gamma(E_-)$,
$$g_1([X,Y],Z)=-g_1(Y,[X,Z])-Yg_1(X,Z)+g_1(X,[Y,Z])$$
and
$$g_1([X,Y],W)=Xg_1(Y,W)-g_1(Y,[X,W])-g_1(X,[Y,W]).$$
That is, $\left(\mathcal{L}_Yg_1\right)(X,Z)+g_1(Y,[X,Z])=0$ and
$({\mathcal L}_Xg_1)(Y,W)-g_1(X,[Y,W])=0$;
\item Finally, let us show that $\lbra L_{++},L_{-+}\rbra\subseteq \Gamma_{g_2}$.\\\\
For all $X\in \Gamma(E_+)$ and $Y\in \Gamma(E_-)$ we have that 
$$\lbra X+g_1(X),Y-g_1(Y)\rbra=[X,Y]-i_Xd(g_1(Y))-i_Yd(g_1(X))$$
is a section of $\Gamma_{g_1}$ if and only if, for all $Z\in \Gamma(E_+)$ and all $W\in \Gamma(E_-)$,
$$g_1([X,Y],Z)=g_1(Y,[X,Z])-Yg_1(X,Z)+g_1(X,[Y,Z])$$
and
$$-g_1([X,Y],W)=-Xg_1(Y,W)+g_1(Y,[X,W])+g_1(X,[Y,W]).$$
That is, $\left({\mathcal L}_Yg_1\right)(X,Z)-g_1(Y,[X,Z])=0$ and $\left({\mathcal L}_Xg_1\right)(Y,W)-g_1(X,[Y,W])=0$.
\end{enumerate}
\end{proof}
\begin{rem}
While the distributions $E_{\pm}$ are involutive if and only if $(g_1, T)$ is a \emph{metric product structure}, the pair $\phi$ is uniformly integrable if and only if  $(g_1, T)$ is a metric product structure and additionally the metric $g_1$ is \emph{bundle-like} with respect to each foliation $E_{\pm}$. Moreover, the bundle-like conditions are equivalent to
$$\nabla_{\Gamma(E_-)}\Gamma(E_+)\subseteq \Gamma(E_+),\qquad \nabla_{\Gamma(E_+)}\Gamma(E_-)\subseteq \Gamma(E_-),$$
where $\nabla$ is the (uniquely defined) Levi-Civita connection of the metric $g_1$.
\end{rem}

\subsection{Uniformly integrable $N$-tuples of type $(-1,-1,\dots,-1)$}
In this subsection, we will study some properties of uniformly integrable $N$-tuples of type $(-1,-1,\dots,-1)$, and then apply it to the case $N=2$. More specifically, we will consider the case when $\phi_1,\phi_2$ correspond to the Jordan-Chevalley decomposition of a skew-symmetric endomorphism $J$. 

In the following proposition, the symbol $a^{\top}b$ will denote the \lq\lq dot product\rq\rq\ $\sum_ia_ib_i$ of two $N$-tuples $a,b$. The result below links the uniform integrability of a commuting family $\phi$ of skew endomorphisms to the integrability of the linear combinations $c^{\top}\phi$.

\begin{prop} If $\varepsilon=(-1,-1,\dots,-1)$, then, $c^{\top}\phi$ is uniformly integrable (that is, its Courant-Nijenhuis torsion vanishes) for all $c\in\mathbb{R}^N$ if and only if the tensors
$${\mathcal T}^{[s]}_{\phi}:=\sum_{s'\in [s]}{\mathcal T}_{\phi}^{s'}$$
vanish, for all $S^3$-orbits $[s]\in\{1,2,3,\dots,N\}^3/S_3$. In particular, if $\phi$ is uniformly integrable, then so are the endomorphisms $c^{\top}\phi$.

\end{prop}
\begin{proof} By an iterated application of Proposition \ref{prop: properties}, for all $P\in\mathcal{R}_1$, we have
\begin{align*}
P\bullet_{c^{\top}\phi}\tau_C&=P(x_1+\dots+x_N,y_1+\dots+y_N,z_1+\dots+z_N)\bullet_{{\rm diag}(c)\phi}\tau_C\\
&=P(c^{\top}x,c^{\top}y,c^{\top}z)\bullet_{\phi}\tau_C .
\end{align*}
Now, let $P(x,y,z)=(x+y)(y+z)(z+x)=x^2(y+z)+y^2(x+z)+z^2(x+y)+2xyz$. Then
$$P(c^{\top}x,c^{\top}y,c^{\top}z)=$$$$=\sum_{i,j,k}c_ic_jc_k\left(x_ix_k(y_j+z_j)+y_iy_j(x_k+z_k)+z_jz_k(x_i+y_i)+x_iy_jz_k+x_ky_iz_j\right)=$$
$$=\sum_{i,j,k}c_ic_jc_kT_{\epsilon}^{ijk}(x,y,z)=\sum_{[s]}c_sT_{\epsilon}^{[s]}(x,y,z),$$
where we denote
$$T^{[s]}_{\epsilon}:=\sum_{s'\in[s]}T_{\epsilon}^{s'}$$
and $c_s:=c_ic_ic_j$ for all $s=(i,j,k)\in \{1,2,3,\dots, N\}^3$. It follows that the Courant-Nijenhuis torsion of  $c^{\top}\phi$ equals
$$\mathcal{S}_{c^{\top}\phi}=\sum_{[s]}c_s{\mathcal T}_{\phi}^{[s]}.$$
Finally, note that the above summation vanishes for all $c\in\R^N$ if and only if $\mathcal{T}_{\phi}^{[s]}=0$ for all $s\in\{1,2,\dots, N\}^3$, proving the result.
\end{proof}
\begin{rem} Note that for all $\sigma\in S_3$ we have $\sigma T^s=T^{\sigma's}$ where $\sigma'=(132)\sigma(123)$. This implies that $\sigma T^{[s]}=T^{[s]}$ and from Proposition \ref{prop: poly properties} and Remark \ref{rem: alternating} follows that the tensors $\mathcal{T}^{[s]}_{\phi}$ are alternating.
\end{rem}
\begin{example}Let $J:\T M\rightarrow \T M$ be a skew endomorphism for which the Jordan decomposition holds, i.e.\ $J=\phi_1+\phi_2$ where $\phi=(\phi_1, \phi_2)$ is a commuting pair of signature $(-1,-1)$ with $\phi_1$ semisimple and $\phi_2$ nilpotent. In general, the integrability  of $J$ is independent of the simultaneous integrability of $\phi_1$ and $\phi_2$. In fact, the Courant-Nijenhuis torsion splits as follows:
$$\mathcal{S}_J=\mathcal{T}_{\phi}^{[111]}+\mathcal{T}_{\phi}^{[112]}+\mathcal{T}_{\phi}^{[122]}+\mathcal{T}_{\phi}^{[222]}=$$
$$=\mathcal{S}_{\phi_1}+\mathcal{T}_{\phi}^{[112]}+\mathcal{T}_{\phi}^{[122]}+\mathcal{S}_{\phi_2}.$$
Therefore, if $\mathcal{T}_{\phi}^{[112]}+\mathcal{T}_{\phi}^{[122]}=0$, the integrability of $\phi_1$ and $\phi_2$ implies the integrability of $J$. This fact is  analogous to a result obtained in \cite{AldiGrandini22} which links the \emph{minimality} of $J$ to the minimality of $\phi_1$ and $\phi_2$. In particular, the tensor
$\mathcal{T}_{\phi}^{[112]}+\mathcal{T}_{\phi}^{[122]}$ plays a role analogous of the \emph{bivariate Courant-Nijenhuis torsion} of $\phi_1$ and $\phi_2$, which is a tensor provided that $\phi_1\phi_2=0$. 

In fact, we can find an explicit relation between the bivariate torsion and  $\mathcal{T}_{\phi}^{[112]}+\mathcal{T}_{\phi}^{[122]}$. Indeed, the bivariate torsion is given by 
$$\mathcal{B}_{\phi}=\left[(x_1+z_1)(y_2+z_2)+(x_2+z_2)(y_1+z_1)\right]\bullet_{\phi}\tau_C$$
so that
$$(x_1+y_1)\bullet_{\phi}\mathcal{B}_{\phi}=\mathcal{T}_{\phi}^{121}+\mathcal{T}_{\phi}^{112}$$
and 
$$(x_2+y_2)\bullet_{\phi}\mathcal{B}_{\phi}=\mathcal{T}_{\phi}^{221}+\mathcal{T}_{\phi}^{212}.$$
Also,
$$z_1\bullet_{\phi}\mathcal{B}_{\phi}=\mathcal{T}_{\phi}^{211}+\left[z_1z_2(x_1+y_1)-x_1x_2(y_1+z_1)-y_1y_2(x_1+z_1)+2z_1^2z_2\right]\bullet_{\phi}\tau_C$$
and 
$$z_2\bullet_{\phi}\mathcal{B}_{\phi}=\mathcal{T}_{\phi}^{122}+\left[z_1z_2(x_2+y_2)-x_1x_2(y_2+z_2)-y_1y_2(x_2+z_2)+2z_1z_2^2\right]\bullet_{\phi}\tau_C\,.$$
Assuming $\phi_1\phi_2=0$ we obtain
$$(x_1+y_1+z_1+x_2+y_2+z_2)\bullet_{\phi}\mathcal{B}_{\phi}=\mathcal{T}_{\phi}^{[112]}+\mathcal{T}_{\phi}^{[122]}.$$
\end{example}

A commuting family $\phi$ of skew endomorphisms satisfying $\mathcal{T}^{[s]}_{\phi}=0$ for all $s \in \{1,2,3,\dots,N\}^3/S_3$ need not be uniformly integrable in general. However, this property holds under additional assumptions.

For instance, let $N=2$ and $M=G$ be a Lie group with Lie algebra $\mathfrak{g}$. The space of left-invariant sections of $\T G$ is the Drinfeld double $D(\g):=\g\oplus\g^*$. Moreover, $D(\g)$ is a Lie algebra with respect the Courant-Dorfman bracket, and the restriction of $\tau_C$ to $D(\g)^{\otimes 3}$ is alternating. Furthermore, assume that all endomorphisms $\phi_i$ preserve $D(\g)$. Then, by Proposition \ref{prop: poly properties}, for all $s\in\{1,2,3,\dots N\}$ and $\sigma\in S_3$,
$$\sigma\mathcal{T_\phi}^{s}{\big|_{D(\g)^{\otimes 3}}}=(-1)^{\sigma}\cdot\mathcal{T_\phi}^{\sigma's}{\big|_{D(\g)^{\otimes 3}}}$$
where $\sigma'=(132)\sigma (123)$. This implies, for example, that if $\mathcal{T_\phi}^{s}=0$ for some $s$, then $\mathcal{T_\phi}^{\sigma s}=0$ for all $\sigma\in S_3$. Furthermore, it implies that $\mathcal{T_\phi}^{[s]}=0$. Suppose $N=2$ and assume $\mathcal{T_\phi}^{[112]}=0$ i.e.\ 
$$\mathcal{T_\phi}^{112}=-\mathcal{T_\phi}^{121}-\mathcal{T_\phi}^{211}\,.$$
Then
$$\mathcal{T_\phi}^{112}{\big|_{D(\g)^{\otimes 3}}}=(12)\mathcal{T_\phi}^{112}{\big|_{D(\g)^{\otimes 3}}}+(13)\mathcal{T_\phi}^{112}{\big|_{D(\g)^{\otimes 3}}}$$
which implies that $\mathcal{T_\phi}^{112}=0=\mathcal{T_\phi}^{121}=\mathcal{T_\phi}^{211}$. Similarly,
$\mathcal{T_\phi}^{[122]}=0$ implies
$$\mathcal{T_\phi}^{122}=0=\mathcal{T_\phi}^{212}=\mathcal{T_\phi}^{221}\,.$$
This means that a pair of left invariant skew endomorphisms $(\phi_1,\phi_2)$ of $\T G$ is uniformly integrable if and only if all linear combinations $c_1\phi_1+c_2\phi_2$ have vanishing shifted Courant-Nijenhuis torsion.

\appendix
\section{An Alternate Approach Using Gr\"obner Basis Techniques}\label{app: alternate}

In this appendix, we will explore an alternate technique that could be used to compute $\mathcal{I}_\varepsilon$. We use this method to show that the set of generators $P^{ij}_{\varepsilon}$ and $T^{ijk}_{\varepsilon}$ form a Gr\"obner basis for the signature $\varepsilon=(1,1,\ldots,1)$, allowing us deduce other properties of the ideal coming from the Gr\"obner degeneration. 

\subsection{Gr\"obner Squeeze Theorem }

We will begin with a simplification that reduces the computation of a generating set for the intersection of $I_\varepsilon^x, I_\varepsilon^y$ and $I_\varepsilon^z$ in terms of a Gr\"obner basis computation for the intersection of any two of the three ideals. We start by providing the required technical results to make this reduction precise, and then we proceed to compute specific cases in the next subsection.

\begin{lem}\label{lem: Grobner_subset}\cite[Problem 2.8]{EH} Let $<$ be a monomial order on a polynomial ring $S$. Suppose that $I\subseteq J$ are ideals of $S$ such that $\init_<(I)=\init_<(J)$. Then $I=J$.    
\end{lem}

Suppose that we have an intersection of ideals $I_1\cap\cdots\cap I_r$ for which we do not have a generating set, and a conjectured Gr\"obner basis for that intersection, say $\mathcal{G} = \{ g_1,\ldots, g_n\}$. If we can find another ideal $J$ which contains $I_1\cap\cdots\cap I_r$ and has an initial ideal equal to the monomial ideal generated by the initial terms of the $g_i$, then $\mathcal{G}$ is a Gr\"obner basis for the intersection. The next proposition, in the special cases where it applies, allows us to bypass the need to compute $S$-polynomials for $I_1\cap\cdots\cap I_r$ if we can find a clever containment of ideals for which computing the Gr\"obner basis is easier. 

\begin{prop}[Gr\"obner Squeeze Theorem]\label{prop: squeeze theorem}
Let $<$ be a monomial order on a polynomial ring $S$. Let $\mathcal{G} = \{ g_1,\ldots, g_n\} \subset I_1\cap\cdots\cap I_r\subset S$. Suppose that $J_1,\ldots J_k$ are ideals in $S$ such that $I_1\cap\cdots\cap I_r\subseteq J_1 \cap \ldots \cap J_k$ and for which 
\[\langle \init_<(g_1),\ldots,\init_<(g_n)\rangle = \bigcap_{i=1}^k\init_<(J_i).\]

\noindent Then $\mathcal{G}$ is a Gr\"obner basis of $I_1\cap\cdots\cap I_r$ with respect to $<$. 

\end{prop}
\begin{proof}

We have the following containment of ideals

\[ \langle \mathcal{G}\rangle \subseteq I_1\cap\cdots\cap I_r \subseteq J_1 \cap \ldots \cap J_k\]

\noindent and \[\langle \init_<(g_1),\ldots,\init_<(g_n)\rangle \subseteq \init_<(\langle \mathcal{G}\rangle)\subseteq \init_<(I_1\cap\cdots\cap I_r)\subseteq \init_<(J_1 \cap \ldots \cap J_k)\subseteq  \bigcap_{i=1}^k\init_<(J_i).\] 

Since $\langle \init_<(g_1),\ldots,\init_<(g_n)\rangle = \bigcap_{i=1}^k\init_<(J_i)$, we must have equality between all initial ideals in the above containment. In particular,  $\init_<(\langle \mathcal{G}\rangle )=\init_<(I_1\cap\cdots\cap I_r)$. By Lemma \ref{lem: Grobner_subset}, $\langle \mathcal{G}\rangle = I_1\cap\cdots\cap I_r$. Furthermore, $\mathcal{G}$ is a Gr\"obner basis of $I_1\cap\cdots\cap I_r$ since $\init_<(I_1\cap\cdots\cap I_r) = \langle\init_<(g_1),\ldots,\init_<(g_n)\rangle$.

\end{proof}

\subsection{A Gr\"obner basis for the case $\varepsilon =(1,1,\ldots,1)$}

We now turn our attention to computing a Gr\"obner basis for $\mathcal{I}_\varepsilon$ for the case $\varepsilon =(1,1,\ldots,1)$ using the technique from the last section. To apply Proposition \ref{prop: squeeze theorem}, we will first need to compute a Gr\"obner basis for $I_\varepsilon^xI_\varepsilon^y$, $I_\varepsilon^xI_\varepsilon^z$ and $I_\varepsilon^yI_\varepsilon^z$. Let us give notation for the generators first introduced in Section \ref{sec: tensorial}. For $i=1,\ldots, N$, define 
\[ f_i:= y_i-z_i \hspace{5mm} g_i:= z_i-x_i \hspace{5mm} h_i:= x_i-y_i.\]

For the purposes of this section, it will be easier to work with the monomial order $\prec_N$ on $\mathcal{R}$ defined by $x_1\succ_N x_2\succ_N\ldots \succ_N x_N\succ_N y_1\succ_N\ldots\succ_N y_N\succ_N z_1\succ_N \ldots \succ_N z_N$. Then $I_\varepsilon^xI_\varepsilon^y$, $I_\varepsilon^xI_\varepsilon^z$ and $I_\varepsilon^yI_\varepsilon^z$ are each naturally generated by $\{f_ig_j\},\{f_ih_j\}$ and $\{g_ih_j\}$ respectively, for $i,j=1,\ldots, N$. Two of these sets are already a Gr\"obner basis for the ideal that they generate. Using this we intend to show that the generating set $T^{ijk}_\varepsilon$ and $P^{ij}_\varepsilon$ from Theorem \ref{thm: reduce to n=6} is a Gr\"obner basis of $\mathcal{I}_\varepsilon$ with respect to $\prec_N$ for $\varepsilon =(1,1,\ldots,1)$ .

With respect to $\prec_N$, it is not difficult to compute 

\[\init_{\prec_N}(T^{ijk}_\varepsilon) = x_ix_ky_j\]

\[\init_{\prec_N}(P^{ij}_\varepsilon) = x_iy_j \text{ if $i<j$}.
\] 

\noindent Similarly, 

\[\init_{\prec_N}(f_ig_j)= -x_jy_i \hspace{1cm}  \init_{\prec_N}(f_ih_j)= x_jy_i \hspace{1cm}  \init_{\prec_N}(g_ih_j) = -x_ix_j. \]

\begin{prop}\label{prop: GB_gens2}
With respect to the monomial order $\prec_N$, the following are Gr\"obner bases for each product ideal:
\[ I_\varepsilon^xI_\varepsilon^y = \langle f_ig_j| 1\leq i,j \leq N\rangle\]
\[ I_\varepsilon^xI_\varepsilon^z = \langle f_ih_j| 1\leq i,j \leq N\rangle\]
\[ I_\varepsilon^yI_\varepsilon^z = \langle g_ih_j, P^{ij}_\varepsilon| 1\leq i,j \leq N\rangle \]

\end{prop}
\begin{proof}
Checking that all $S$-pairs reduce to $0$ is left as an easy (albeit tedious) exercise through several cases. It is also useful noting that in general, if the initial terms between two polynomials are relatively prime, then their $S$-pair automatically reduces to $0$. For instance, $S(g_ih_j, g_ih_k) = g_iS(h_j,h_k)$ and $S(h_j,h_k)$ reduces to $0$ using the above fact. This leaves only a few cases to be checked by hand. One could also use the techniques from Section \ref{subsec: Knutson} (in particular from Remark \ref{rem: Knutson Grobner}) to prove the first two claims.
\end{proof}

\begin{cor}\label{cor: initial ideal intersections}
The initial ideal, with respect to $\prec_N$, for each ideal in Proposition \ref{prop: GB_gens2} is:
\[\init_{\prec_N}(I_\varepsilon^xI_\varepsilon^y) = \langle x_iy_j | 1\leq i,j \leq N\rangle\]
\[\init_{\prec_N}(I_\varepsilon^xI_\varepsilon^z) = \langle x_iy_j | 1\leq i,j \leq N\rangle \]
\[\init_{\prec_N}(I_\varepsilon^yI_\varepsilon^z) =  \langle x_ix_j, x_ky_\ell | 1\leq i,j,k,\ell \leq N , k<\ell\rangle \]

\noindent Furthermore, the intersection $\init_{\prec_N}(I_\varepsilon^xI_\varepsilon^y) \cap \init_{\prec_N}(I_\varepsilon^xI_\varepsilon^z) \cap \init_{\prec_N}(I_\varepsilon^yI_\varepsilon^z) $ is minimally generated by 
\[ \langle x_ix_jy_k,x_\ell y_m | 1\leq i,j,k,\ell,m\leq N, \ell<m, k\leq i,j\rangle.\]

\end{cor}

The intersection $\init_{\prec_N}(I_\varepsilon^xI_\varepsilon^y) \cap \init_{\prec_N}(I_\varepsilon^xI_\varepsilon^z) \cap \init_{\prec_N}(I_\varepsilon^yI_\varepsilon^z)$ is straightforward to compute. In fact, we don't need $\init_{\prec_N}(I_\varepsilon^xI_\varepsilon^y) $ in the intersection since it equals $\init_{\prec_N}(I_\varepsilon^xI_\varepsilon^z) $. We include it since other signatures might require the intersection of all three ideals.

Also note that if $k>i$ or $k>j$, then the monomial $x_iy_k$ or $x_jy_k$ is in the intersection, so the generator $x_ix_jy_k$ is redundant. Therefore, to get the unique minimal monomial generating set, we impose the restriction $k\leq i,j$. We are finally ready to show the main result.

\begin{theorem}\label{thm: product_grobner}
Fix $\varepsilon=(1,1,\ldots,1)$ and let $I_\varepsilon^x, I_\varepsilon^y,$ and $I_\varepsilon^z$ be ideals of $\mathcal{R}$ as defined as in Section \ref{sec: tensorial}, endowed with the monomial order $\prec_N$. Then the set
\[\mathcal{G}_\epsilon= \{P^{ij}_{\varepsilon},T^{ijk}_{\varepsilon}|1\leq i,j,k\leq N\} \]

\noindent is a Gr\"obner basis of $\mathcal{I}_{\varepsilon}=I_\varepsilon^x\cap I_\varepsilon^y\cap I_\varepsilon^z$ with respect to $\prec_N$. 
\end{theorem}

\begin{proof}

Since $\init_{\prec_N}(\mathcal{G}_{\varepsilon})= \{ x_ix_jy_k,x_\ell y_m | 1\leq i,j,k,\ell,m\leq N, \ell<m\}$, by Corollary \ref{cor: initial ideal intersections}, $\init_{\prec_N}(\mathcal{G}_{\varepsilon})=\init_{\prec_N}(I_\varepsilon^xI_\varepsilon^y) \cap \init_{\prec_N}(I_\varepsilon^xI_\varepsilon^z) \cap \init_{\prec_N}(I_\varepsilon^yI_\varepsilon^z)$. The result follows by Proposition \ref{prop: squeeze theorem} applied to $J_1=I_\varepsilon^x\cap I_\varepsilon^y =I_\varepsilon^xI_\varepsilon^y $, $J_2=I_\varepsilon^x\cap I_\varepsilon^z =I_\varepsilon^xI_\varepsilon^z $ and $J_3=I_\varepsilon^y\cap I_\varepsilon^z =I_\varepsilon^yI_\varepsilon^z $.

\end{proof}

\begin{rem}
Using Theorem \ref{thm: product_grobner}, we can immediately recover the fact that $\mathcal{I}_\varepsilon$ is a radical ideal since $\init_{\prec_N}(\mathcal{I}_\varepsilon)$ is a square-free monomial ideal. We could also use information from  $\init_{\prec_N}(\mathcal{I}_\varepsilon)$ to check other properties of $\mathcal{I}_\varepsilon$, like being equidimensional, Cohen-Macaulay, or other properties that are open in flat families. See \cite[Chapter 15]{E} for more information on Gr\"obner degeneration. 
\end{rem}

\section*{Acknowledgements} This work is supported in part by VCU Quest Award ``Quantum Fields and Knots: An Integrative Approach.'' Da Silva's research is supported by NSF LEAPS-MPS Grant 2532757.

\end{document}